\documentclass[12pt]{article}
\usepackage{cancel}
\usepackage{tabularx}
\usepackage{tikz}
\usepackage{relsize}
mm \textheight=8.9in

\usepackage{amssymb}
\usepackage{amsmath}
\usepackage{amsthm}
\usepackage{color}

\usepackage{tikz}
\usetikzlibrary{patterns,snakes}

\usepackage{tikz-cd}

\usepackage{array,multirow}

\newcommand{\br}[3]{{$#1$}$\lower4pt\hbox{$\tp\atop\raise4pt \hbox{$\scriptscriptstyle{#2}$}$} ${$#3$}}
\newcommand{\tw}[3]{{$#1$}${\,\scriptscriptstyle {#2}}\atop\raise9pt\hbox{$\scriptstyle\tp$} ${$#3$}}
\newcommand{\ttps}[2]{{#1}\raise5pt\hbox{$\lower12pt\hbox{$\scriptstyle\tp$}\atop \lower0pt\hbox{$\tilde\;$}$}\raise4.5pt\hbox{${\scriptstyle{#2}}$}}
\newcommand{\st}[1]{\mbox{${\,\scriptscriptstyle {#1}}\atop\raise5.5pt\hbox{$*$}$}}

\newcommand{\rd}[1]{\mbox{${\,\scriptscriptstyle {#1}}\atop\raise5.5pt\hbox{$\bullet$}$}}
\newcommand{\rt}[1]{\otimes_\chi}
\newcommand{\lt}[1]{\mbox{${\,\scriptscriptstyle {#1}}\atop\raise5.5pt\hbox{$\ltimes$}$}}
\newcommand{\btr}{\raise1.2pt\hbox{$\scriptstyle\blacktriangleright$}\hspace{2pt}}
\newcommand{\btl}{\raise1.2pt\hbox{$\scriptstyle\blacktriangleleft$}\hspace{2pt}}

\newcommand{\lcr}{\raise1.0pt \hbox{${\scriptstyle\rightharpoonup}$}}
\newcommand{\rcr}{\raise1.0pt \hbox{${\scriptstyle\leftharpoonup}$}}

\newcommand{\ttp}{{\lower12pt\hbox{$\tp$}\atop \hbox{$\tilde\;$}}}

\newcommand{\id}{\mathrm{id}}

\newcommand{\Ic}{\mathcal{I}}

\newcommand{\Ag}{\mathfrak{A}}
\newcommand{\Dg}{\mathfrak{D}}
\newcommand{\Bg}{\mathfrak{B}}

\newcommand{\Cg}{\mathfrak{C}}

\newcommand{\C}{\mathbb{C}}
\newcommand{\Z}{\mathbb{Z}}

\newcommand{\tp}{\otimes}

\newcommand{\zt}{\zeta}

\newcommand{\ve}{\varepsilon}
\newcommand{\gm}{\gamma}
\newcommand{\dt}{\delta}

\newcommand{\op}{\oplus}
\newcommand{\la}{\lambda}
\newcommand{\tr}{\triangleright}

\newcommand{\End}{\mathrm{End}}

\newcommand{\Span}{\mathrm{Span}}
\newcommand{\Aut}{\mathrm{Aut}}

\newcommand{\rk}{\mathrm{rk}}

\newcommand{\Rm}{\mathrm{R}}

\newcommand{\La}{\Lambda}

\newcommand{\mg}{\mathfrak{m}}
\newcommand{\g}{\mathfrak{g}}
\renewcommand{\b}{\mathfrak{b}}
\renewcommand{\k}{\mathfrak{k}}

\newcommand{\h}{\mathfrak{h}}

\newcommand{\s}{\mathfrak{s}}

\renewcommand{\o}{\mathfrak{o}}
\newcommand{\n}{\mathbf{n}}

\newcommand{\m}{\mathbf{m}}
\newcommand{\eps}{\epsilon}

\newcommand{\nn}{\nonumber}
\newcommand{\p}{\mathfrak{p}}
\renewcommand{\l}{\mathfrak{l}}
\renewcommand{\c}{\mathfrak{c}}

\newcommand{\si}{\sigma}
\newcommand{\al}{\alpha}
\renewcommand{\t}{\mathfrak{t}}
\newcommand{\bt}{\beta}

\newcommand{\be}{\begin{eqnarray}}
\newcommand{\ee}{\end{eqnarray}}

\newtheorem{thm}{Theorem}[section]
\newtheorem{propn}[thm]{Proposition}
\newtheorem{lemma}[thm]{Lemma}
\newtheorem{corollary}[thm]{Corollary}

\newtheorem{remark}[thm]{Remark}
\newtheorem{definition}[thm]{Definition}
\newtheorem{example}[thm]{Example}

\newcount\prg

\newcommand{\parag}{\advance\prg by1 {\noindent\bf\thesection.\the\prg\hspace{6pt}}}

\begin{document}

\title{Graded Satake diagrams and super-symmetric pairs}
\author{D. Algethami${}^{\dag}$, A. Mudrov${}^{\sharp}$, V. Stukopin${}^\sharp$
\vspace{10pt}\\
\small ${\dag}$ Department of Mathematics, College of Science,\\
\small
 University of Bisha, P.O. Box 551, Bisha 61922, Saudi Arabia,
\vspace{10pt}\\
\small
${\sharp}$ Moscow Institute of Physics and Technology,\\
\small
9 Institutskiy per., Dolgoprudny, Moscow Region,
141701, Russia,
\vspace{10pt}\\
\small
 e-mail: dalgethami@ub.edu.sa,  mudrov.ai@mipt.ru, stukopin.va@mipt.ru
\vspace{10pt}
\\
}

\maketitle

\begin{abstract}
We list classical spherical subalgebras in basic matrix Lie superalgebras which are quantizable to coideal subalgebras in the standard
 quantum supergroups, for any choice of Borel subalgebra.
 We classify the corresponding Satake-type diagrams and prove that each of them
 defines a family of proper spherical subalgebras.
\end{abstract}

\\


\newpage
\tableofcontents

\section{Introduction}


This article is a continuation of our recent work \cite{AMS}  on  spherical subalgebras in  Lie superalgebras,
where we addressed  a special choice of Borel subalgebra.
Homogeneous spherical manifolds generalize symmetric spaces \cite{VK} and are locally represented by pairs of a reductive total Lie algebra and a stabilizer subalgebra.
The non-graded classical geometry of such manifolds has been a textbook topic \cite{Hel}, while its super-symmetry analogue is a relatively new theme (see, for example  \cite{Sh,Sh1}).
The invention of quantum groups \cite{FRT,D1} naturally brings about their non-commutative variants of  homogeneous and, in particular, symmetric  spaces
 \cite{NS,NDS}. In  its modern form, the theory of quantum symmetric pairs  was developed by  Letzter in \cite{Let}.
 Nowadays it has been evolved to a vibrant  chapter of quantum algebra \cite{K,BK,RV,AV}. Its supersymmetric generalization
 appears to be a natural further step down that path.

Unlike for the non-graded reducible Lie algebras, different Borel subalgebras in Lie superalgebras are in general not isomorphic.
It is therefore meaningful to study different polarizations because the mere definition of spherical subalgebra depends on a chosen Borel subalgebra.
Another argument is that non-isomorphic polarizations lead to different quantum groups. The class of spherical subalgebras under the current study
comprises exactly those which admit quantization along Letzter's lines.

Like in  our previous work \cite{AMS}, we study basic matrix Lie superalgebras: general linear and ortho-symplectic. They constitute
the bulk case of finite dimensional reductive Lie superalgebras. However,  we confined ourselves in \cite{AMS} with a special case of symmetric grading of their natural module and with the  minimal number of
odd simple roots in the Borel subalgebra. In this paper we drop this restriction and give a full analysis  to all possible polarizations of the
total Lie superalgebra.

In this presentation we focus on the quasiclassical theory and do not address such important issues as coideal subalgebras \cite{Let},
R-matrices, reflection equation \cite{KSkl, KSS} {\em etc},
which are points of common interest in the quantum version of the theory. The current  work lays a foundation for further studies in this field as it rounds up a variety
of target objects.
The main finding of the paper is a classification of super-symmetric  pairs relative to a given polarization. Such a  classification
 can be encoded in graded diagrams of Satake type, similarly to the traditional non-graded approach. It should be noted that the list of superspherical subalgebras is significantly richer than
 the corresponding list for their non-graded analogs even in the minimal symmetric setting of \cite{AMS}.

The problem of quantum super-symmetric pairs was also addressed in  \cite{Shen, ShenWang} from a different angle and under restriction
to only even Levi core subalgebras.
That approach made use of bosonization of quantum supergroups {\em a la} Radford-Majid, and Yamane's theory of Lusztig automorphisms.
We employ a straightforward application of the root theory and a concept of Weyl operator acting on roots and weights.
That is an analog of the longest element of the Weyl group of a non-graded root system. This element plays a key role in the conventional
theory of symmetric pairs, both classical and quantum. In the super-symmetric case, the Weyl group is
too small to accommodate an element with  required properties, but fortunately it admits  a substitute   at least for  basic matrix Lie superalgebras.
In fact, our construction implicitly refers to the Weyl groupoid, however we do not go far along this path.
What was special to \cite{AMS} is that  we considered only even Weyl operators (which preserve a partition to even and odd roots).
However this assumption is too restrictive and has to be dropped if, say, the highest and lowest vectors of the modules involved have  different degrees. Such a  modification  allows to cover all Borel subalgebras and extends the theory specifically for the  general linear Lie superalgebra.

It is worthy to note that our approach has certain  similarities with  \cite{RV} dealing with  the non-graded case.
However our logic is quite inverse to that of  \cite{RV}. We do not start with an involutive  automorphism $\tau$ of the Cartan matrix
subject to a bi-partition of simple roots. Instead, we  formulate quasiclassical conditions on such a split to generate a quantizable subalgebra,
like we did in \cite{AMS}. In the case of even Weyl operator, these two lines of  reasoning turned out to be equivalent.
If the Weyl operator is not even, our approach  pays off, because the involution $\tau$ fails to be an automorphism.
Rather, it gives rise to  weird Satake diagrams, which might be more consistently  regarded as pairs of non-isomorphic decorated Dynkin
graphs.

The paper is organized as follows.
Section \ref{Sec_ClSPhP} develops a general theory of classical super-symmetric pairs, and relates them to
decorated Dynkin diagrams (DDD). Those are pre-Satake diagrams, each of which is encoding a family of super-spherical subalgebras $\k\subset \g$,
depending on a vector of mixture parameters. They amount to Satake diagrams that withstand
selection rules listed in Section \ref{Sec_ClSPhP}.  Those rules are arranged  in a set of lemmas which  discard diagrams producing $\k=\g$ for all values of mixture parameters.
We state them without proof referring to \cite{AMS} for details.
The resulting classification of Satake diagrams is given in Section \ref{Sec_GrSD}.
In the last Section \ref{Sec_Nontriviality}, we demonstrate that all Satake diagrams  are non-trivial, i.e. they
define proper $\k\subset \g$ for a non-empty set of vectors of mixture parameters.
We do it by showing that such $\k$ have more invariants than $\g$, in certain $\g$-modules.
These invariants may be viewed as classical analogs of K-matrices.

\section{Classical super-symmetric pairs}
\label{Sec_ClSPhP}
In this section we define  Lie superalgebras $\k\subset \g$ that give rise to
coideal subalgebras in generalization of the Letzter theory to quantum supergroups.
\subsection{Basic classical matrix Lie superalgebras}
\label{Sec_BClLieSAlg}
Let $\g$ be either a general linear or ortho-symplectic  Lie superalgebra. Denote by $V=\C^N$ the natural module of $\g$.
We consider all possible polarizations (triangular decompositions) $\g=\g_-\op \h\op \g_+$, where $\h$ is a  Cartan subalgebra and
$\b_\pm=\h\op \g_\pm$ are Borel subalgebras with nil-radicals $\g_\pm$. The Cartan subalgebra is represented on $V$ by diagonal matrices while
$\b_+$ and $\b_-$ by upper and lower triangular matrices, respectively.

Polarizations  of $\g$ are   induced
by $\Z_2$-gradings of $V$, whose standard weight basis $\{v_i\}_{i=1}^N$ consists of homogeneous vectors $v_i$ of degree $\bar i\in \{0,1\}$.
The weights of $v_i$ are denoted by $\zt_i$. In the ortho-symplectic case, they are subject to condition $\zt_{i'}=-\zt_i$, where $i'=N+1-i$.
Thus, for  ortho-symplectic $\g$ of odd $N$, the weight $\zt_{\frac{N+1}{2}}$ is zero, and $\deg(v_\frac{N+1}{2})=0$.
The weights $\zt_i$ and $\zt_j$ are pairwise orthogonal unless $i =j$ and, for ortho-symplectic $\g$,  $i=j'$.
Furthermore, $(\zt_i,\zt_i)=(-1)^{\bar i}$, for  $i\not =\frac{N+1}{2}$.

Denote by $n$ the number of simple roots of $\g$, which equals $N-1$ for general linear $\g$ and $\left [\frac{N}{2}\right]$ otherwise.
Within the given polarization, the simple positive roots are
$$
     \Pi_{\g\l}=\{\zt_i-\zt_{i+1}\}_{i=1}^{N-1}, \quad
     \Pi_{\s\p\o}=\{\zt_i-\zt_{i+1}\}_{i=1}^{n-1}\cup \{2\zt_n\},
$$
$$
    \Pi_{\o\s\p}=\{\zt_i-\zt_{i+1}\}_{i=1}^{n-1}\cup\{\zt_n\}, \quad \mbox{odd}\>  N,  \quad
     \Pi_{\o\s\p}=\{\zt_i-\zt_{i+1}\}_{i=1}^{n-2}\cup\{\zt_{n-1}\pm \zt_n\},       \quad \mbox{even}\>N.
$$
Simple roots form a basis for the root system $\Rm$, which splits to the subsets of positive and negative roots $\Rm^-\cup \Rm^+$,
with the inclusion $\Pi\subset \Rm^+$.
The basic weights $\{\zt_i\}_{i=1}^N$ for general linear $\g$, and $\{\zt_i\}_{i=1}^n$ otherwise, generate the weight lattice $\La=\La_\g$.

Although the algebra $\s\p\o$ is isomorphic to certain $\o\s\p$ of even $N$, the triangular decompositions are different.
We will also use a notation of $\ve_i$ for even weights of the module $V$ and $\dt_i$ for odd.

Like in the theory of simple Lie algebras, the properties of the root basis are encoded in a Dynkin diagram $D_\g$ with a convention
that isotropic odd roots are coloured grey while non-isotropic are black.
Below are examples
of Dynkin diagrams with minimal number of odd roots corresponding to a symmetric grading $\bar i'=\bar i$, $i=1,\ldots, N$ of the module $V$.
\begin{center}
\begin{picture}(330,30)
\put(0,10){\circle{3}}
\put(30,10){\circle{3}}

\put(80,10){\circle{3}}
\put(110,10){\color{gray}\circle*{3}}
\put(140,10){\circle{3}}

\put(82,10){\line(1,0){26}}
\put(112,10){\line(1,0){26}}

\put(142,10){\line(1,0){10}}
\put(178,10){\line(1,0){10}}
\put(160,10){$\ldots$}
\put(190,10){\circle{3}}
\put(220,10){\color{gray}\circle*{3}}
\put(250,10){\circle{3}}

\put(192,10){\line(1,0){26}}
\put(222,10){\line(1,0){26}}

\put(302,10){\line(1,0){26}}

\put(300,10){\circle{3}}
\put(330,10){\circle{3}}

\put(252,10){\line(1,0){10}}
\put(288,10){\line(1,0){10}}
\put(270,10){$\ldots$}

\put(2,10){\line(1,0){26}}
\put(32,10){\line(1,0){10}}
\put(68,10){\line(1,0){10}}
\put(50,10){$\ldots$}

\put(55,14){\smaller[2]$\delta_{\m-1}-\delta_{\m}$}

\put(-15,14){\smaller[2]$\delta_1-\delta_2$}
\put(15,1){\smaller[2]$\delta_2-\delta_3$}

\put(95,1){\smaller[2]$\delta_{\m}-\ve_1$}
\put(125,14){\smaller[2]$\ve_1-\ve_2$}
\put(167,14){\smaller[2]$\ve_{N-1}-\ve_{N}$}

\put(205,1) {\smaller[2]$\ve_{N}-\delta_{\m+1}$}
\put(230,14) {\smaller[2]$\delta_{\m+1}-\delta_{\m+2}$}

\put(280,1) {\smaller[2]$\delta_{2\m-2}-\delta_{2\m-1}$}
\put(310,14) {\smaller[2]$\delta_{2\m-1}-\delta_{2\m}$}

 \end{picture}
\end{center}
\begin{center}
\begin{picture}(200,30)
\put(0,10){\circle{3}}
\put(30,10){\circle{3}}

\put(80,10){\circle{3}}
\put(110,10){\color{gray}\circle*{3}}
\put(140,10){\circle{3}}

\put(82,10){\line(1,0){26}}
\put(112,10){\line(1,0){26}}

\put(142,10){\line(1,0){10}}
\put(178,10){\line(1,0){10}}
\put(160,10){$\ldots$}
\put(190,10){\circle{3}}

\put(2,10){\line(1,0){26}}
\put(32,10){\line(1,0){10}}
\put(68,10){\line(1,0){10}}
\put(50,10){$\ldots$}

\put(55,14){\smaller[2]$\delta_{\m-1}-\delta_{\m}$}

\put(-15,14){\smaller[2]$\delta_1-\delta_2$}
\put(15,1){\smaller[2]$\delta_2-\delta_3$}

\put(94,1){\smaller[2]$\delta_{\m}-\ve_1$}
\put(125,14){\smaller[2]$\ve_1-\ve_2$}
\put(167,14){\smaller[2]$\ve_{\n-1}-\ve_{\n}$}

\put(220,1) {\smaller[2]$\ve_{\n}$}

\put(220,10){\circle{3}}

\put(191,8.5){\line(1,0){24}}
\put(191,11.5){\line(1,0){24}}
\put(211,7){$>$}
 \end{picture}
\end{center}
  \begin{center}

\begin{picture}(120,30)
\put(0,10){\circle{3}}
\put(1.5,10){\line(1,0){27}}
\put(32,10){\line(1,0){10}}
\put(68,10){\line(1,0){10}}
\put(30,10){\circle{3}}
\put(80,10){\circle{3}}
\put(47,10){$\ldots$}

\put(110,10){\circle*{3}}

\put(81,8.5){\line(1,0){24.5}}
\put(81,11.5){\line(1,0){24.5}}
\put(101.5,7){$>$}

\put(107,14){\smaller[2]$\delta_\m$}

\put(0,14){\smaller[2]$\delta_1-\delta_2$}
\put(20,0){\smaller[2]$\delta_2-\delta_3$}
\put(70,0){\smaller[2]$\delta_{\m-1}-\delta_{\m}$}
 \end{picture}
\end{center}

\begin{center}
\begin{picture}(200,30)
\put(0,10){\circle{3}}
\put(30,10){\circle{3}}

\put(80,10){\circle{3}}
\put(110,10){\color{gray}\circle*{3}}
\put(140,10){\circle{3}}

\put(82,10){\line(1,0){26}}
\put(112,10){\line(1,0){26}}

\put(142,10){\line(1,0){10}}
\put(178,10){\line(1,0){10}}
\put(160,10){$\ldots$}
\put(190,10){\circle{3}}

\put(2,10){\line(1,0){26}}
\put(32,10){\line(1,0){10}}
\put(68,10){\line(1,0){10}}
\put(50,10){$\ldots$}

\put(55,14){\smaller[2]$\delta_{\m-1}-\delta_{\m}$}

\put(-15,14){\smaller[2]$\delta_1-\delta_2$}
\put(15,1){\smaller[2]$\delta_2-\delta_3$}

\put(94,1){\smaller[2]$\delta_{\m}-\ve_1$}
\put(125,14){\smaller[2]$\ve_1-\ve_2$}
\put(153,1){\smaller[2]$\ve_{\n-2}-\ve_{\n-1}$}

\put(191.5,11.5){\line(3,2){25}}
\put(191.5,8.5){\line(3,-2){25}}
\put(217.5,29){\circle{3}}
\put(217.5,-9){\circle{3}}

\put(222,-9){\smaller[2]$\ve_{\n-1}-\ve_{\n}$}
\put(222,29){\smaller[2]$\ve_{\n-1}+\ve_{\n}$}
 \end{picture}
\end{center}

\begin{center}
\begin{picture}(100,60)

\put(0,30){\circle{3}}
\put(1.5,30){\line(1,0){27}}

\put(32,30){\line(1,0){10}}
\put(68,30){\line(1,0){10}}
\put(30,30){\circle{3}}
\put(47,30){$\ldots$}
\put(80,30){\circle{3}}
\put(81.5,31.5){\line(3,2){25}}
\put(81.5,28.5){\line(3,-2){25}}
\put(107.5,49){\color{gray}\circle*{3}}
\put(107.5,11){\color{gray}\circle*{3}}

\put(-10,20){\smaller[2]$\delta_1-\delta_2$}
\put(20,35){\smaller[2]$\delta_2-\delta_3$}
\put(50,20){\smaller[2]$\delta_{\m-1}-\delta_\m$}
\put(112,8){\smaller[2]$\delta_\m+\ve_1$}
\put(112,48){\smaller[2]$\delta_\m-\ve_1$}

\put(106.5,12){\line(0,1){36}}
\put(108.5,12){\line(0,1){36}}

\put(170,27)

\end{picture}
\end{center}

\begin{center}
\begin{picture}(200,30)
\put(0,10){\circle{3}}
\put(30,10){\circle{3}}

\put(80,10){\circle{3}}
\put(94,1){\smaller[2]$\delta_{\m}-\ve_1$}
\put(110,10){\color{gray}\circle*{3}}
\put(140,10){\circle{3}}

\put(82,10){\line(1,0){26}}
\put(112,10){\line(1,0){26}}

\put(142,10){\line(1,0){10}}
\put(178,10){\line(1,0){10}}
\put(160,10){$\ldots$}
\put(190,10){\circle{3}}

\put(2,10){\line(1,0){26}}
\put(32,10){\line(1,0){10}}
\put(68,10){\line(1,0){10}}
\put(50,10){$\ldots$}

\put(55,14){\smaller[2]$\delta_{\m-1}-\delta_{\m}$}

\put(-15,14){\smaller[2]$\delta_1-\delta_2$}
\put(15,1){\smaller[2]$\delta_2-\delta_3$}

\put(125,14){\smaller[2]$\ve_1-\ve_2$}
\put(167,14){\smaller[2]$\ve_{\n-1}-\ve_{\n}$}

\put(220,1) {\smaller[2]$2\ve_{\n}$}

\put(220,10){\circle{3}}

\put(194,8.5){\line(1,0){24.5}}
\put(194,11.5){\line(1,0){24.5}}
\put(190,7){$<$}
 \end{picture}
\end{center}
Note with care that topologically isomorphic Dynkin diagrams do not imply isomorphism of Lie superalgebras, see, for instance  $\g\l(2|2)$  and $\o\s\p(4|2)$
(distinguished polarizations with one odd simple root).

Further on  we drop the parity colour  convention in order to avoid conflicts with additional data inherent to Satake diagrams.
The odd nodes will be depicted with squares while the circles will be reserved for even nodes. A node that may carry arbitrary parity
will be denoted with rhombus.

A diagram $D_\g$  with discarded parity of nodes is a valid  non-graded Dynkin diagram (for
odd tail roots of even $\o\s\p$ we also remove double linking arcs). We call such a non-graded diagram the shape of $D_\g$.
Thus we have four different shapes of graded Dynkin diagrams:  $\Ag$, $\Bg$, $\Cg$, and $\Dg$.
By  tail subalgebra $\t\subset \g$ of shape $\Bg,\Cg,\Dg$ we understand the one with the set of simple roots
$\{\zt_n\}$, $\{2\zt_n\}$, and $\{\zt_{n-1}\pm \zt_n\}$, respectively. By  shaft subalgebra $\s\subset \g$ we mean
the one of type $\Ag$ whose simple root basis is complementary to the tail.
We always keep orientation of the total Dynkin diagram placing the tail sub-graph on the right. This convention does not apply
to sub-diagrams, which are understood up to isomorphism of the corresponding subalgebras, like in the  formulation of the selection rules in Section \ref{SecDDD-SR}.

Changing the grading on $V$ to its opposite does not affect $\g$ and its polarization unless $\g$ is odd ortho-symplectic.
 In that case, the grading is fixed by the parity of the tail root.
If it is odd, then it is not isotropic (it is black under the standard convention).
In all other cases odd simple roots are isotropic. The tail root of shape $\Cg$ is always even.

The ortho-symplectic Lie superalgebra $\g$ is defined as the one preserving a bilinear  form
$$
C=\sum_{k=1}^n (e_{k,k'}+\eps_{k} e_{k',k})+e_{\frac{N+1}{2},\frac{N+1}{2}},
$$
where the last term is present only if $N$ is odd.
The coefficients $\eps_{k}$ take values $\pm 1$ depending on the type of $\g$ and its polarization. They are subject to condition $\eps_{i}\eps_{k}=(-1)^{\bar i+\bar k}$.
The matrix $C$ is even; it is invariant under the $\g$-action
$$
\rho(x)C+C\rho^t(x), \quad x\in \g,
$$
where $t$ is the matrix super-transposition
$A^t_{ij}=(-1)^{(\bar i+\bar j)\bar j}A_{ji}$.
This operation is a super-involutive  anti-automorphism of the graded matrix algebra $\End(V)$.

Root vectors of  orthosymplectic $\g\subset \End(V)$ preserving $C$ can be taken in the form
$$
e_{\zt_k-\zt_m}=-(-1)^{\bar m(\bar k+\bar m)} e_{k,m}+ e_{m',k'}, \quad e_{\zt_k+\zt_m}=-\eps_k (-1)^{\bar k(\bar k+\bar m)} e_{k,m'}+e_{m,k'},
$$
$$
f_{\zt_k-\zt_m}=-(-1)^{\bar k(\bar k+\bar m)} e_{m,k}+ e_{k',m'}, \quad f_{\zt_k+\zt_m}=-\eps_{m}(-1)^{\bar k(\bar k+\bar m)} e_{m',k}+e_{k',m},
$$
for  $1\leqslant k<m\leqslant n$,
and
$$
 e_{2\zt_k}=e_{k,k'},  \quad  f_{2\zt_k}=e_{k',k}, \quad \eps_k=-1, \quad 1\leqslant k\leqslant n,
$$
for even $N=2n$ and
$$
e_{\zt_k}=- e_{k,n+1}+e_{n+1,k'}, \quad f_{\zt_k}=-(-1)^{\bar k} e_{n+1,k}+e_{k',n+1}, \quad \eps_k=(-1)^{\bar k}, \quad 1\leqslant k\leqslant n,
$$
for odd $N=2n+1$.
The Cartan subalgebra is represented by diagonal matrices, while $\g_\pm$ by strictly upper (lower) triangular matrices,  as well as for the general linear $\g$.

\subsection{Weyl operator and spherical data }
\label{SecWOp}
 Let $\g$ be a Lie superalgebra that features a triangular decomposition with Cartan subalgebra $\h$,  and let  $\b\subset \g$ be a Borel subalgebra
 containing $\h$.
\begin{definition}
A Lie super-algebra $\k\subset \g$ is called spherical if $\g=\k+\b$.
Then the pair $(\g,\k)$ is called spherical.
\end{definition}
It is known that, depending on the polarization
 Borel subalgebras in $\g$ are generally  not isomorphic.
Thus, contrary to the non-graded case, this definition of sphericity substantially depends
upon a choice of $\b$.
From now on we restrict our consideration to the case when $\g$ is either general linear or ortho-symplectic.
The choice of $\b$ is determined by a grading of the underlying natural module.

\begin{definition}
\label{Weyl-operator}
 A unique $\Z$-linear map $w_\g\colon \La\to \La$ defined by the assignment
$\zt_i\mapsto \zt_{i'}$, $i=1,\ldots, N,$ is called
Weyl operator.
\end{definition}
Contrary to \cite{AMS}, we do not restrict ourselves to even $w_\g$ (preserving parity of the roots) and  allow for an arbitrary grading of the underlying vector space $V$.
For $\g$ of type $\Ag$ that implies that the highest and lowest weights of $V$ have the same parity if and only if the number of odd simple roots is even.
Otherwise the highest (lowest) weights of $V$ and its dual $V^*$ have opposite parities and $w_\g$ is not even.

Denote by $\breve{\Pi}$ the basis of simple roots of the shape diagram of $D_\g$.
 \begin{lemma}
\label{transposition}
The Weyl operator  $w_\g$  preserves the  weight lattice $\La$,  and the root system $\Rm$.
 Furthermore, $-w_\g(\breve{ \Pi})=\breve{ \Pi}$.
\end{lemma}
\begin{proof}
$-w_\g$-invariance of $\La$ and $\breve{ \Pi}$ is due to the very construction. With regard to $\Rm$, the assertion
follows from the explicit description of the root systems given in Section \ref{Sec_BClLieSAlg}.
\end{proof}
\noindent
Remark that $-w_\g=\id$ in the case of ortho-symplectic (symplecto-orthogonal) $\g$, therefore $-w_\g(\Pi)=\Pi$ in this case.
For general linear $\g$, the operator $-w_\g$, in general, produces a different, although topologically isomorphic
Dynkin diagram amounting to an  isomorphic Borel subalgebra. Yet it induces a permutation on $\breve{ \Pi}$. We denote it by
$-\breve{ w}_\g$.


\begin{definition}
  The grading of the underlying module $V$ is called symmetric if $\deg(v_i)=\deg(v_{i'})$, for all $i=1,\ldots,N$.
\end{definition}
\noindent
We also call symmetric the  induced polarization of $\g$.
In such a polarization, the Weyl operator preserves the parity of weights and roots and furthermore $-w_\g(\Pi)=\Pi$.
The grading is always symmetric for $\g$ of type $\Bg,\Cg,\Dg$.

Pick a subset $\Pi_\l\subset \Pi$,  put  $\bar \Pi_\l=\Pi \backslash \Pi_\l$, and generate a  subalgebra $\l=\langle e_\al,f_\al\rangle_{\al\in \Pi_\l}\subset \g$.
It is a direct sum of subalgebras, $\l=\sum_{i}\l_i$, corresponding to connected components of $\Pi_\l$.
If $\l_i$ is of type $\Ag$, then we set $\hat \l_i\subset \g$ to be the natural $\g\l$-extension of $\l_i$, and leave $\hat \l_i=\l_i$ otherwise.
Denote by $\hat \l=\op_i \hat\l_i\subset \g$. Its Cartan subalgebra and dual are denoted by $\h_{\hat\l}$ and $\h_{\hat\l}^*$, respectively.
The restriction of the canonical inner product from $\h^*$ to $\h^*_{\hat \l}$ is non-degenerate.

For the $i$-th connected component of $\Pi_\l$ put  $w_{\l_i}=w_{\hat \l_i}$ and  extend it to   $\La$ as  identical on the orthogonal complement
 to $\La_{\hat \l_i}$.
We  define the Weyl operator $w_\l\in \End(\La)$  of the subalgebra $\hat \l$
as $w_\l=\prod_i w_{\l_i}$.
Clearly $w_\l$ is involutive and preserves the weight lattice and root system of $\l$.
We use the same symbol to denote the  extension of  $w_\l$ to $\h^*=\La\tp_\Z \C$.
\begin{definition}
The subset  $\Pi_\l$ and the subalgebra $\l$ are called regular if $w_\l$ preserves the root system $\Rm_\g$.
\end{definition}
\noindent
Otherwise the subalgebra $\l$ and subset $\Pi_\l$ are called irregular.

For general linear $\g$, every connected $D_\l$ is regular. For ortho-symplectic $\g$, irregularity may also occur in the shaft.

\begin{propn}

Suppose that $D_\l$ is connected. Then $\l$ is irregular precisely when $\g$ is ortho-symplectic, $D_\l$ is of type $\Ag$, and the polarization of $\l$ is not symmetric.

\end{propn}
\begin{proof}

For a component of type $\Ag$, the operator $w_\l$ reverses its basic weights. In the ortho-symplectic case this preserves $\Rm_\g$ exactly when it preserves their parities, which is equivalent to symmetric polarization. Components containing the full tail are regular.

\end{proof}
\noindent
Note with care that $w_\l$ is not an isometry in general.

Consider vector subspaces
$$\mathfrak{m}_+=\sum_{\al\in \Rm^+_{\g}\>- \>\Rm^+_{\l}} \g_\al,
\quad
\mathfrak{m}_-=\sum_{\al\in \Rm^+_{\g}\>- \>\Rm^+_{\l}} \g_{-\al},
$$
as graded $\l$-modules.
For $\al\in \bar \Pi_\l$ let $V^\pm_\al\subset \m_\pm$ denote the $\l$-submodule generated by $e_{\pm \al}\in \g_{\pm \al}$.
It is a tensor product of modules over  the subalgebras in $\l$ corresponding to the connected components in $D_\l$: the ones
of minimal dimension or its skew-(super)symmetrized tensor square over a component of general linear type, see \cite{Serg}.
\begin{lemma}
\label{high-low}
Suppose that $\l$ is regular. Then
\begin{itemize}
  \item[i)]For each $\al\in \bar \Pi_\l$, the $\l$-module $V^\pm_\al$ is irreducible and determined by its lowest (highest) weight.
  \item[ii)] The operator $w_\l$ flips  the highest and lowest weights of $V^\pm_\al$ for each $\al \in \bar \Pi_\l$.
\end{itemize}
\end{lemma}
\begin{proof}
The subalgebra $\l$ is a sum of $\l=\sum_{i} \l_i$ of basic Lie superalgebras $\l_i$ corresponding to
connected components of the Dynkin diagram $D_\l$.
In all cases excepting $\g\in \Dg$, the modules $V^\pm_\al$ are tensor products of smallest fundamental $\l_i$-modules,
for which the lemma is obviously true.

We are left to consider  $\g\in \Dg$ and we may assume that $D_\l$ is connected.
The module which is not minimal for $\l$  may occur in the following   two cases.

a) If $\Pi_{\l}=\{\al_{n-2},\al_{n-1},\al_{n}\}$, then the non-trivial $\l$-submodule in $\mathfrak{m}_\pm $  is $V^\pm_{\al_{n-3}}\simeq \C^6$, for which the statement is obvious.

b) Another possibility is when $\al_{n-1}\in \Pi_{\l}$, $\al_{n}\in \bar \Pi_\l$
(or the other way around) and the polarization of $\l$ is  symmetric. Then the highest  weight in  $V^+_{\al_n}$
is $\zt_1+\zt_2$, see Section \ref{Sec_irreg}, while the lowest is $\al_n=\zt_{n-1}+\zt_n$. They are flipped by $w_\l$, which proves ii).
This module is also known to be irreducible, self-dual, and determined by its highest weight.  This proves i).
 \end{proof}

Suppose that  $\tau \in \Aut(\bar \Pi_\l)$  is a permutation and let $\tilde \al$ denote the highest weight of $V_{\tau(\al)}^+$.
Suppose that $\tilde \al$ and $\al$ have the same parity.
\begin{definition}
\label{triple}
  The triple $(\g,\l, \tau)$ is called super-symmetric if
\be
(\mu+\tilde\mu,\al)&=0,& \quad \forall \mu\in \bar \Pi_\l, \quad \al\in \Pi_{\l},
\label{1nd-cond}\\
(\mu+\tilde\mu,\nu-\tilde \nu)&=0,& \quad \forall \mu,\nu\in \bar \Pi_\l.
\label{2nd-cond}
\ee
\end{definition}
\noindent
These identities admit the following algebraic interpretation.
Condition  (\ref{1nd-cond}) means that the root vectors $e_\mu$ and $f_{\tilde \mu}$ have the same
transformation properties under the adjoint action of $\l$, by Lemma \ref{high-low}. Their linear combination generates
a submodule $\simeq V^+_\mu$.
Condition (\ref{2nd-cond}) is needed for quantization in order to secure coideal properties of subalgebras in quantum groups associated with super-symmetric triples.

Our goal is to find all $(\Pi_\l,\tau)$ solving the system (\ref{1nd-cond}--\ref{2nd-cond}). The Weyl operator introduced above has been
specifically devised for this task. However, it features the required properties only for regular $\Pi_\l$.
The case of irregular $\Pi_\l$ will be addressed directly in Section \ref{Sec_irreg}. Until then
we assume that $\Pi_\l$ is regular.
Then we can  extend $\tau$ as a $\Z$-linear automorphism of the root lattice $Q_\g=\Z\Pi$.
We can also define $\breve{\tau}\in \Aut(\breve{\Pi})$ as  $-\breve{ w}_\l$ on $\breve{ \Pi}_\l$
and coinciding with $\tau$ on $\bar\Pi_\l$ naturally identified with a subset of $\breve{\Pi}$ complementary to $\breve{ \Pi}_\l$.

By Lemma \ref{high-low}, we can set $\tilde \al =w_\l\circ \tau(\al)\in \Rm^+$ for all $\al\in \bar \Pi_\l$.
\begin{lemma}
\label{tau-commutes w}
Suppose that $V^+_\al\simeq V^-_{\tau(\al)}$ for all $\al\in \bar \Pi_\l$. Then  $\tau$ commutes with $w_\l$ as a $\Z$-linear endomorphism of the root lattice.
\end{lemma}
\begin{proof}
We reproduce the proof of the analogous statement for an even $w_\l$ given in \cite{AMS}.
By construction, $\tau$ and $w_\l$ commute when restricted to $\h^*_\l$.
It suffices to check that also for simple roots from $\bar \Pi_\l$.

By Lemma \ref{high-low},  $w_\l$ takes the highest weight of an irreducible $\l$-module $V^\pm_\mu$, $\mu\in \bar \Pi_\l$, to the lowest weight and vice versa.
For each $\mu\in \bar \Pi_\l$ we have $w_\l(\mu)=\mu+\eta$ for some weight $\eta\in \Z_+\Pi_\l$ that satisfies
 $w_\l(\eta)=-\eta$ because $w_\l^2=\id$. The weight $\eta$ depends only on the projection of $\tau(\mu)$ to $\h^*_\l$,
therefore
$- \tau(\mu)=w_\l(-\tilde \mu)=w_\l\bigl(-\tau(\mu)\bigr)+\eta$ or
$\tau(\mu)+\eta=w_\l\bigl(\tau(\mu)\bigr)$. Here we used the hypothesis of the lemma.
Then, since $\tau(\eta)=-w_\l(\eta)$ for all $\eta\in \h^*_\l$,
$$
\tau\bigl(w_\l(\mu)\bigr)=\tau(\mu+\eta)=\tau(\mu)+\tau(\eta)=\tau(\mu)-w_\l(\eta)=\tau(\mu)+\eta=w_\l\bigl(\tau(\mu)\bigr),
$$
as required.
\end{proof}
\noindent
Remark that the hypothesis of this lemma is equivalent to condition (\ref{1nd-cond}) thanks to Lemma \ref{high-low} i).

Define a linear map $\theta=-w_\l\circ \tau$ on $Q_\g\otimes_\Z\C$. It preserves $\Rm$ because so do $\tau$ and $w_\l$.
 Furthermore,
$\theta(\al)=-\tilde \al$ for $\al \in \bar \Pi_\l$
and $\theta(\al)=\al$ for $\al \in \Pi_\l$.
 The system of equalities  (\ref{1nd-cond}) and (\ref{2nd-cond})
 translates to
\be
\label{gen-cond}
\bigl(\al+\theta(\al),\bt-\theta(\bt)\bigr)=0,\quad \forall \al,\bt\in \Pi.
\ee

\begin{propn}
\label{tau-invol-orth}
Condition (\ref{gen-cond})
is  fulfilled if and only if the operator $\tau$ is involutive, commutes with $w_\l$, and the composition  $\theta=-w_\l\circ \tau$ extends to
 an involutive isometry on $Q_\g\otimes_\Z\C$.
\end{propn}
\begin{proof}
First of all note that a linear operator being orthogonal and involutive is the same as being symmetric and involutive,
or orthogonal and symmetric simultaneously.

Denote by $Q_\g=\Z \Pi$ the root lattice of $\g$.
Condition (\ref{gen-cond}) is bilinear and therefore holds true for any pair of vectors from $Q_\g\otimes_\Z\C$.
Setting $\al=\bt$ in  (\ref{gen-cond}) we find that $\theta$ is an isometry.
Then
 (\ref{gen-cond})
 translates to
$$
\bigl(\theta(\al), \bt\bigr)=\bigl(\al,\theta(\bt)\bigr), \quad \forall \al, \bt\in \Pi.
$$
It means that  $\theta$ is a symmetric operator. Therefore it is an involutive isometry.

Conversely, suppose that  $\tau$ is  involutive, coincides with $-w_\l$ on $\Pi_\l$, and the composition $\theta=-w_\l\circ \tau$ is involutive and orthogonal.
Then
$$
(\al,\mu)=\bigl(\theta (\al), \mu\bigr)=\bigl(\al,\theta(\mu)\bigr)=-\bigl(\al,w_\l\circ\tau(\mu)\bigr)=-(\al,\tilde \mu)
$$
 for all $\al\in \Pi_\l$ and $\mu\in \bar \Pi_\l$.
Thus, the condition (\ref{1nd-cond}) is fulfilled, and $\tau$ commutes with $w_\l$, by Lemma \ref{high-low} i) and  Lemma \ref{tau-commutes w}.
Since  $\theta=-w_\l\circ \tau$ is orthogonal and involutive,  (\ref{gen-cond}) holds true either.
\end{proof}
The involution $\theta$ preserves the root system $\Rm$. This is obvious if $-w_\l$ preserves $\Pi$. 
Otherwise
$\tau\not =\id$, then the sub-diagram $\breve{\tau}$ is necessarily the flip of the shape diagram $\breve{D}_\g$ of type $\Ag$.
The diagram $\Pi_\l$ is either empty or connected,
see (\ref{GL-I}).
The sub-lattice $\La_\l$ is generated by $\zt_i$, with $m<i< m'$ for some $m\leqslant[\frac{N}{2}]$.
Then $\tau$ extends to $\La_\g$ as $\tau (\zt_i)=-\zt_{i'}$, for all $i=1,\ldots,N$.
The operator $\theta$ is identical on $\La_\l$ and flips $\zt_i\leftrightarrow \zt_{i'}$, $i\leqslant m$.

\noindent
Note that for regular $\l$ our requirement for roots $\al$ and $\tilde \al=-\theta(\al)$,  $\al\in \bar \Pi_\l$,
to be of  the same parity is redundant once $\theta$ satisfies
(\ref{gen-cond}). It is then automatically fulfilled  because  $\theta$ is an isometry.
\begin{remark}
\label{even-isometry}
\em
If the mapping $w_\l$ is even (preserves the parity of weights and roots of $\l$), then its
extension to an operator on $\h^*$  is an isometry. Then condition (\ref{gen-cond})
is fulfilled if and only if $\tau$ is an involutive isometry preserving $\Pi$ and
therefore an automorphism of Dynkin diagram.
By definition, $w_\l$ is even if and only if the polarization of $\l$ is symmetric.

  Contrary to the case of symmetric grading, $\tau$ is not an automorphism of the Dynkin diagram of $\g$. Rather, it is
an isomorphism between two generally different diagrams. Indeed,  $D_\g$ comprises the following data: the sub-diagram $D_\l$ and
the adjoint $\l$-module structure on $\g_\pm$. The operator $-w_\l$ induces an automorphism of $\l$ (possibly changing $D_\l$).
We extend $-w_\l$ to $\tau$ by the requirement that the module structure on $\g_+$ is taken by $\tau$ to  that on $\g_-$.
Thus $\tau(D_\g)$ becomes isomorphic to $D_\g$.
\end{remark}
The necessity of considering a pair of Dynkin diagrams instead of the single one as for even $w_\l$ makes the study more complicated.
It can be  simplified  if we pass to the  set of  basic weights, $W=\{\pm\zt_i\}_{i=1}^N$, because it is preserved by both $w_\l$ and $\tau$.
It splits to a union $W=W_\l\coprod \bar W_\l$, where $W_\l$ comprises the  weights of $\l$ and $\bar W_\l$ is its complement (which is orthogonal to $W_\l$).
Then the operator $\theta$ restricts to an orthogonal involutive  permutation on $\bar W_\l$ and identical on $W_\l$.

\begin{example}
\label{gl(3|1)}
\em
  Take $\g\l(3|1)$ for $\g$ with the following Dynkin diagram:
\be
\begin{picture}(350,15)
\put(160,3){\circle{3}}
\put(161.5,3){\line(1,0){27}}

\put(188.5,1.5){\framebox(3,3)}

\put(191.5,3){\line(1,0){24}}
\put(217,1.5){\framebox(3,3)}

\put(155,8){$\al$}
\put(185,8){$\bt$}

\put(212,8){$\gm$}
 \end{picture}
\ee
where $\al$ is even while the square nodes $\bt$ and $\gm$ are odd. Elements of  $\Pi_\l$ will be painted  black.
In our case,
we take $\Pi_\l$ consisting of a single element $\bt$. We have a pair of isomorphic diagrams
\be
\begin{picture}(80,15)
\put(10,3){\circle{3}}
\put(11.5,3){\line(1,0){27}}

\put(38.5,1.5){\textcolor{black}{\rule{3pt}{3pt}}}

\put(41.5,3){\line(1,0){24}}
\put(66,1.5){\framebox(3,3)}

\put(5,8){$\al$}
\put(35,8){$\bt$}

\put(62,8){$\gm$}
 \end{picture}
\quad
\begin{picture}(80,15)
\put(8,1.5){\framebox(3,3)}
\put(11.5,3){\line(1,0){27}}

\put(38.5,1.5){\textcolor{black}{\rule{3pt}{3pt}}}

\put(41.5,3){\line(1,0){24}}
\put(67,3){\circle{3}}

\put(5,8){$\gm'$}
\put(35,8){$\bt'$}

\put(62,8){$\al'$}
 \end{picture}
\ee
where $\mu'=\tau(\mu)$ for $\mu=\al,\bt,\gm$.
In terms of the original diagram, the operator $\tau$ fixes  $\bt$ and permutes $\al$ and $\gm$.
It is not even because $w_\l$ is not even on weights. The reason is that the lowest ($\al$) and  highest ($\al+\bt$) weights of the 2-dimensional
$\l$-module $V^+_\al$ have different parities. Therefore $V^+_\al\simeq V^-_\gm$ only if $\al$ and $\gm$ have different parities.

Let us look at this example in terms of the basic weights  $W_\g=\{\pm\ve_1,\pm\ve_2,\pm\dt,\pm\ve_3\}$.
The Weyl operator flips $\ve_2\leftrightarrow \dt$. The operator $\tau$ acts by
$$
\tau\colon\{\ve_1,\ve_2,\dt,\ve_3\}\mapsto \{-\ve_3,-\dt,-\ve_2,-\ve_1\}
$$
while the operator $\theta$ reads:
$$
\theta\colon \{\ve_1,\ve_2,\dt,\ve_3\}\mapsto \{\ve_3,\ve_2,\dt,\ve_1\}.
$$
It is clearly an involutive isometry.
\end{example}

Suppose that the pair $(\Pi_\l,\tau)$  solves the system (\ref{1nd-cond}--\ref{2nd-cond}).
Let $\c\subset \h$ denote the centralizer of $\l$ in $\h$.
For each $\al\in \bar \Pi_\l$ pick  $c_\al\in \C^\times $, $\grave c_{\al}\in \C$,  and $u_\al\in \c$ assuming $u_\al\not =0$ only if $\al$ is even,  orthogonal to $\Pi_\l$, and $\tilde  \al=\al=\tau(\al)$. For each $\la\in \h^*$ let $h_\la\in \h$  denote the dual element to $\la$ with respect to the inner product on $\h^*$.
Put
\begin{equation}
\begin{array}{rcl}
y_\al &=& h_\al - h_{\tilde \al}, \\
x_\al &=& e_\al + c_\al f_{\tilde \al} + \grave c_{\al} u_\al,
\end{array}
\label{gen-spher-superpairs}
\end{equation}
for all $\al\in \bar \Pi_\l$.
Define a Lie subalgebra $\k\subset \g$ as the one generated by $\l$ and by $x_\al$, $y_\al$ with $\al\in  \bar \Pi_\l$.
\begin{definition}
  The pair of Lie superalgebras $\k\subset \g$ determined by a super-symmetric triple $(\g,\l,\tau)$ and by $c_\al, \grave{c}_\al,u_\al$, $\al \in \bar \Pi_\l$, is called super-symmetric.
\end{definition}
\noindent
The  complex numbers $c_\al,\grave c_{\al}$ in (\ref{gen-spher-superpairs}) are called mixture parameters. We will put all $\grave c_{\al}$ to zero,
for the sake of simplicity. By a vector of the mixture parameters we will understand a set $\vec c= (c_\al)_{\al\in \bar \Pi_\l}$ with non-zero components.
Thus the subalgebra $\k$ is determined by the triple $(\Pi_\l,\tau, \vec c)$.

Let us explain the conditions $\tilde  \al=\al=\tau(\al)$ on the appearance of $u_\al$ in the $x_\al$.
Since $\h$ consists of even elements, both $\al$ and $\tilde \al$ must be even. Furthermore, $\al$ and $\tilde \al$ must be orthogonal to $\l$
as $u_\al$ is in its centralizer. Finally, the requirement
$$
[y_\al,x_\al]=\bigl((\al,\al)-(\al,\tilde\al)\bigr)e_\al+c_\al\bigl((\tilde \al,\tilde \al)-(\tilde \al, \al)\bigr)f_{\tilde\al}\propto x_\al
$$
forces
$$
(\al,\al)=(\al,\tilde\al)=(\tilde \al,\tilde \al),
$$
which is possible only if $\al=\tilde \al$.
\begin{propn}
\label{ps-sym=sup-sph}
  A Lie superalgebra $\k$ determined by a super-symmetric triple and a vector of mixture parameters is spherical.
\end{propn}
\begin{proof}
See \cite{AMS}.
\end{proof}
\noindent
Remark that generators $y_\al$ are needed so that the coproducts of the positive and negative components of $x_\al$ satisfy the coideal condition in the quantization $U_q(\k)\subset U_q(\g)$.

The rest of the paper is devoted to the question when the subalgebra $\k\subset \g$ is proper for a given super-symmetric  triple.
 \subsection{Decorated Dynkin diagrams and selection rules}
\label{SecDDD-SR}
Like in the non-graded case the operator $\tau$ entering a super-symmetric triple $(\g,\l,\tau)$
 can be visualized via decorated Dynkin diagrams (DDD). We identify it with the permutation of the shape of Dynkin diagram.
 Such an identification is correct because it acts as a permutation $-\breve{w}_\l$ on $\breve{\Pi}_\l$. It allows us to
 consider a single diagram instead of pairs of diagrams.

We use black colour for nodes in $\Pi_\l$ and white for $\bar \Pi_\l$ regardless of their parity.
The parity  will be  either described in words or via the following convention:
 circles designate even roots while squares stand for odd; a rhombus means  a root of  arbitrary parity.

As the subalgebra $\k$ is defined  through a set of generators, it is not {\em a priori} obvious when it is proper, i.e. strictly less than $\g$.
Otherwise $\k$ is not interesting, and such a pair $(\g,\k)$ is called trivial.
Below we formulate criteria that rule out DDD  giving rise to trivial pairs  for all values
of the mixture parameters. Such diagrams are considered as trivial and should be discarded.

Selection rules that filter out trivial DDD were formulated  for the minimal symmetric grading of $V$  in \cite{AMS} and they turn out to be sufficient  for a general grading.
We recall  them without proof.

It is convenient for the study of DDD to reduce them to smaller parts.
Let $C(\al)$ denote  the union of connected components of  $D_\l\subset D_\g$ that are  connected to  $\{\al,\tau(\al)\}\subset \bar \Pi_\l$.
 \begin{definition}
A decorated Dynkin  sub-diagram is a subgraph $D'$ in the total Dynkin graph $D$ such that $\tau(\al)\in D'$ and $C(\al)\subset D'$ for every white node $\al \in D'$.
\end{definition}
\noindent
For instance, the graph with nodes $C(\al)\cup \{\al,\tau(\al)\}$ is the minimal decorated sub-diagram that includes $\al\in \bar \Pi_\l$.
We denote it by $D(\al)$. More generally, $D(\al_1,\ldots, \al_k)$ will designate the decorated sub-diagram generated by $\al_1,\ldots, \al_k\in \bar \Pi_\l$.
It is the union of $\al_i,\tau(\al_i)$ and $C(\al_i)$, over $i=1,\ldots, k$. We write $\g^{(\al)}$ for the subalgebra of $\g$ supported on $D(\al)$.

Every decorated sub-diagram $D'\subset D$  defines subalgebras $\g'\subset \g$ and  $\l'\subset \l$, whose simple root generators
are the nodes of $D'$ and $D_\l'\cap D'$, respectively. Given
a spherical subalgebra $\k\subset \g$  determined by $(D=D_\g,D_\l,\tau)$ and a  mixture parameter  vector, we define
$\k'$ as generated by $\l'$ and by (\ref{gen-spher-superpairs}) with all white $\al \in D'$.
Clearly $(\g',\l',\tau')$, where $\tau'=\tau|_{D'}$, is a super-symmetric triple, and $(\g',\k')$ is a spherical pair.

The next statement is a rectification of the only  non-graded selection rule from  \cite{RV}.
\begin{lemma}
\label{non-grad-sel-rule}
Suppose that a decorated Dynkin diagram is such that
\be
\label{RVSR}
D(\bt) \simeq
\begin{picture}(35,10)
\put(2,1){$\scriptstyle\blacklozenge$}
\put(27,1){$\scriptstyle\lozenge$}
 \put(7,3){\line(1,0){21}}
 \put(1,7){$\small \al$} \put(30,7){$\small \bt$}
 \end{picture}
\ee
 for some $\bt\in \bar \Pi_l$.
 Then  $\g^{(\bt)}\subset \k$ unless $\bt$ is odd and $\al$ is even.
 \end{lemma}
\noindent
This lemma indicates that the  $\Z_2$-graded situation is more versatile.
\begin{lemma}
\label{isolated odd}
  Suppose that a decorated Dynkin diagram contains a sub-graph isomorphic to
\be
\label{ISO-ODD}
\begin{picture}(35,10)
\put(2,1){$\scriptstyle\lozenge$}
\put(30.5,1.5){\framebox(3,3)}
\multiput(7,3)(6,0){4}{\line(1,0){3}}
 \put(1,7){$\small \al$} \put(30,7){$\small \bt$}
 \end{picture}
\ee
 where a grey odd node $\{\bt\} =D(\bt)$, and
  $(\bt,\al)\not =0$.  Then $\g^{(\al)}+\g^{(\bt)}\subset \k$.
\end{lemma}
\noindent
Let us emphasise that this statement holds true for any $D(\al)$. It means that a white odd root fixed by $\tau$ cannot be isolated from black roots,
in a connected Dynkin diagram.

The wording for the next lemma is improved compared to \cite{AMS}, where the root $\si$ was unnecessarily stated even.
That was a presentational flaw, as this restriction was used in \cite{AMS} neither in the proof nor in applications of the lemma.
 \begin{lemma}
  \label{le-b-d}
Suppose that a decorated Dynkin diagram contains a sub-graph isomorphic to
\be
\label{4NODES}
\begin{picture}(350,15)
\put(160,3){\circle*{3}}
\put(161.5,3){\line(1,0){27}}

\put(188.5,1.5){\framebox(3,3)}

\put(191.5,3){\line(1,0){24}}
\put(217,3){\circle*{3}}
\put(240,1){$\scriptstyle\lozenge$}

\multiput(217,3)(6,0){4}{\line(1,0){3}}

\put(155,8){$\al$}
\put(185,8){$\bt$}

\put(212,8){$\gm$}
\put(242,8){$\sigma$}
 \end{picture}
\ee
where $D(\bt)=\{\al,\bt,\gm\}$ and $(\gm,\si)\not=0$. Suppose that $\tau(\si)=\si$ and  $\al\not \in D(\si)$. Then $\g^{(\bt)}+\g^{(\si)}\subset \k$.
\end{lemma}
\noindent
Let us emphasise that the graphs (\ref{RVSR}, \ref{ISO-ODD},\ref{4NODES}) are meant up to isomorphism (regardless of their orientation on the plane).

 The next lemma specially addresses Dynkin diagrams of  even orthogonal shape.
\begin{lemma}
\label{sel-rul-d}
 Suppose that a decorated Dynkin diagram of shape  $\Dg$  contains one of the sub-diagrams
\be
\label{D-TAIL}
\begin{picture}(60,10)
\put(0,3){\circle*{3}}
\put(1.5,3){\line(1,0){12}}
\put(13.5,1.5){\framebox(3,3)}
\put(16.5,3){\line(1,0){12}}
\put(30.5,3){\circle*{3}}
\put(32.5,3.5){\line(1,1){10}}
\put(32.5,3.2){\line(1,-1){10}}
\put(42,12.5){\framebox(3,3)}
\put(42,-9.5){\framebox(3,3)}
\put(42,12.5){\line(0,-1){18}}\put(45,12.5){\line(0,-1){18}}
\qbezier(46,-7)(54,3)(46,12)
\put(47,10.5){\vector(-2,3){2}}
\put(47,-5.5){\vector(-2,-3){2}}
\end{picture}
\quad\quad
\begin{picture}(55,10)
\put(0,3){\circle*{3}}
 \put(1.5,3){\line(1,0){12}}
\put(13.5,1.5){\framebox(3,3)}
 \put(16.5,3){\line(1,0){12}}

 \put(30.5,3){\circle*{3}}

\put(32.5,3.5){\line(1,1){10}}\put(32.5,3.2){\line(1,-1){10}}
\put(43,14){\circle{3}}\put(43,-8){\circle{3}}

\qbezier(46,-6)(53,4)(46,13)
\put(47,11.5){\vector(-2,3){2}}\put(47,-4.5){\vector(-2,-3){2}}
\end{picture}
\quad\quad\quad
\begin{picture}(38,10)
\put(0,3){\circle*{3}}
 \put(1.5,3){\line(1,0){12}}

 \put(14,1.5){\framebox(3,3)}

\put(17.5,3.5){\line(1,1){10}}\put(17.5,3.2){\line(1,-1){10}}
\put(28,14){\circle{3}}\put(28,-8){\circle*{3}}

\end{picture}
\ee
 where circles are even and squares are odd.
 Then $\g^{(\bt)}\subset \k$ for each white node $\bt$ here.
\end{lemma}
\begin{proof}
  A new thing as compared to \cite{AMS}, Lemma 4.21,  is the diagram on the left.
  It goes over to the middle diagram by the change of parity for the basic weight $\zt_n$.
  This change does not affect the proof, which is  the same as for the middle diagram in \cite{AMS}.
\end{proof}

\noindent
Informally, Lemmas \ref{non-grad-sel-rule}--\ref{sel-rul-d} mean that
the white nodes in their graphs should be {\em de facto} re-painted as black
(this is however producing a devastating effect on all mixed generators, see Lemma \ref{ruin mixture} below).
Note that Lemma \ref{sel-rul-d} is not applicable  to diagram
$
\>\begin{picture}(30,12)
\put(0,3){\circle{3}}
 \put(1.5,3){\line(1,0){12}}

 \put(14,1.5){\framebox(3,3)}

\put(17.5,3.5){\line(1,1){10}}\put(17.5,3.2){\line(1,-1){10}}
\put(28,14){\circle*{3}}\put(28,-8){\circle*{3}}

\end{picture}
$
although it topologically coincides with one in (\ref{D-TAIL}). These two diagrams generate subalgebras with drastically different properties.
\begin{definition}
We call a triple  $(\g,\l,\tau)$ and the corresponding decorated Dynkin diagram trivial if the subalgebra $\k$
they generate coincides with $\g$ for all values of mixture parameters $c_\al\in \C^\times$, $\grave c_\al \in \C$,  $\al\in \bar \Pi_\l$.
\end{definition}

We will say that a decorated diagram $D$ violates selection rules if either
 $D$ contains a sub-diagram (\ref{RVSR}) distinct from
$
\simeq \begin{picture}(15,10)
\put(2,3){\circle*{3}}
\put(3,3){\line(1,0){7}}
\put(10.5,1.5){\framebox(3,3)}
\end{picture}
$
 or one of the sub-diagrams (\ref{ISO-ODD}), (\ref{4NODES}), (\ref{D-TAIL}), or the configuration of Lemma \ref{sel-rul-d}.

The mechanism facilitating   selection rules boils down to the following observation.
\begin{lemma}
\label{ruin mixture}
  A spherical pair $(\g,\k)$ is trivial if and only if $\g^{(\al)}\subset \k$ for some $\al\in \bar \Pi_\l$.
\end{lemma}
\begin{proof}
  Only if is obvious. The converse is proved in a similar way to  Proposition 4.24 stated in  \cite{AMS}
  for the minimal symmetric grading.
\end{proof}
\begin{corollary}
\label{violating SR implies triviality}
  A decorated Dynkin diagram is trivial if it violates selection rules.
\end{corollary}
\noindent
Thus a triple $(\g,\l,\tau)$ (respectively, the decorated Dynkin diagram) is trivial if for each vector of mixture parameters there is
a white simple root $\al$ such that root vectors $e_\al$ and $f_\al$ belong to $\k$.
The converse to Corollary \ref{violating SR implies triviality} is also true. We postpone its proof to Section \ref{Sec_Nontriviality}.
\begin{definition}
  Decorated Dynkin diagrams that obey the selection rules are called (graded) Satake diagrams.
\end{definition}
\noindent
It is clear that if a DDD is Satake, then its every sub-diagram is Satake too.
We reserve the letter $D$ to denote Dynkin graphs while a given  Satake diagram supported on $D$ will be denoted
by $S$. This will also apply to sub-diagrams generated by a subset of white nodes.

Original non-graded Satake diagrams and their generalizations parameterize certain involutive automorphisms of root systems, see e.g. \cite{RV}.
 A similar interpretation
works for their super-symmetric analogs with even Weyl operator \cite{AMS}.  Next section extends this view
due to  "weird" Satake diagrams, or pairs of DDD, that come into play for general super-symmetric pairs.
 \section{Super Satake diagrams}
\label{Sec_GrSD}
In this section we describe decorated Dynkin diagrams that obey the selection rules. In a subsequent section we will prove that they are all non-trivial.
Recall that we adopt the following convention: simple roots from $\Pi_\l$ are depicted by black nodes,  while those  from $\bar \Pi_\l$ by white.
Even nodes are circles, odd nodes are squares. If a node can be of any parity, we denote it with rhombus.
The number $|\Pi_\l|$ of black nodes in the diagram is called its black rank. The cardinality $|\bar \Pi_\l|$ is called its white rank.

\subsection{General linear $\g$}
\subsubsection{Nonidentical $\tau$}
Suppose first that $\tau\not =\id$.
If $\Pi_\l=\varnothing$ and  $\rk\>\g =2m-1$, then $\al_m=\tau(\al_m)$ can be only even. The Satake diagram is
shown on the left in (\ref{GL-I}) with suppressed  specification of the parities. The nodes linked with arcs have the same parity.

Suppose that  the  sub-diagram $D_\l\subset D_\g$ contains a connected component of rank 2 or higher.
Then the selection rules tell us that $D_\l$ is connected.
We shall call the polarization of $\l$ even if the number of odd simple roots  in $\Pi_\l$ is even.
Otherwise the polarization is called odd.
\be
\begin{picture}(120,30)
\put(-2,1){$\scriptscriptstyle\lozenge$}\put(1.5,3){\line(1,0){12}}\put(13,1){$\scriptscriptstyle\lozenge$}
\put(16.5,3){\line(1,0){9}}\put(28,0){$\cdots$} \put(44.5,3){\line(1,0){9}}
\put(53,1){$\scriptscriptstyle\lozenge$}\put(56.5,3){\line(1,0){12}}\put(68,1){$\scriptscriptstyle\lozenge$}

\put(-2,27){$\scriptscriptstyle\lozenge$}\put(1.5,29){\line(1,0){12}}\put(13,27){$\scriptscriptstyle\lozenge$}
\put(16.5,29){\line(1,0){9}}\put(28,26){$\cdots$} \put(44.5,29){\line(1,0){9}}
\put(53,27){$\scriptscriptstyle\lozenge$}\put(56.5,29){\line(1,0){12}}\put(68,27){$\scriptscriptstyle\lozenge$}

\put(71.5,29){\line(1,-1){12}}
\put(71.5,3){\line(1,1){12}}\put(84.5,16){\circle{3}}

\put(-2,20){\vector(1,3){2}}\put(-1.8,11){\vector(1,-3){2}}\put(13,20){\vector(1,3){2}}\put(13.2,11){\vector(1,-3){2}}
\put(53,20){\vector(1,3){2}}\put(53.2,11){\vector(1,-3){2}}\put(68,20){\vector(1,3){2}}\put(68.2,11){\vector(1,-3){2}}
\qbezier(0,6)(-5,16)(0,26)\qbezier(15,6)(10,16)(15,26)\qbezier(55,6)(50,16)(55,26)\qbezier(70,6)(65,16)(70,26)
\put(90,15){$\scriptstyle \al_m$}
\end{picture}
\quad
\begin{picture}(80,32)
\put(-2,1){$\scriptscriptstyle\lozenge$}\put(1.5,3){\line(1,0){12}}\put(13,1){$\scriptscriptstyle\lozenge$}
\put(16.5,3){\line(1,0){9}}\put(28,0){$\cdots$} \put(44.5,3){\line(1,0){9}}
\put(53,1){$\scriptscriptstyle\lozenge$}\put(56.5,3){\line(1,0){12}}\put(68,1){$\scriptscriptstyle\blacklozenge$}
\put(70,28){\line(0,-1){5}}\put(70,4){\line(0,1){5}}
\put(68.5,11){\vdots}
\put(-2,27){$\scriptscriptstyle\lozenge$}\put(1.5,29){\line(1,0){12}}\put(13,27){$\scriptscriptstyle\lozenge$}
\put(16.5,29){\line(1,0){9}}\put(28,26){$\cdots$} \put(44.5,29){\line(1,0){9}}
\put(53,27){$\scriptscriptstyle\lozenge$}\put(56.5,29){\line(1,0){12}}\put(68,27){$\scriptscriptstyle\blacklozenge$}

\put(-2,20){\vector(1,3){2}}\put(-1.8,11){\vector(1,-3){2}}\put(13,20){\vector(1,3){2}}\put(13.2,11){\vector(1,-3){2}}
\put(53,20){\vector(1,3){2}}\put(53.2,11){\vector(1,-3){2}}
\qbezier(0,6)(-5,16)(0,26)\qbezier(15,6)(10,16)(15,26)\qbezier(55,6)(50,16)(55,26)
\put(74,14){$\scriptstyle D_\l$}
\put(53,34){$\scriptstyle \al_m$}

\end{picture}
\label{GL-I}
\ee
For an even polarization of $\l$, the adjacent to $D_\l$ white nodes have the same parity, otherwise their parities are different.
All other pairs of nodes connected with arcs have the same parities.
This case can be also extended for  $|\Pi_\l|=0,1$.

Let $\al_m$ be the white root preceding the black block in the diagram on the right (\ref{GL-I}), counting from the left. For empty $\Pi_\l$ and even rank, $m$ is set to $\frac{n}{2}$. The involution $\theta$ acts on the weights of $V$
by the assignment $\theta(\zt_i)= \zt_{i'}$, $i\leqslant m$, and $\theta(\zt_i)=\zt_i$, $m<i<m'$.
The weights $\zt_i$ and $\zt_{i'}$ with $i<m$ have the same parity. The parity of weights $\zt_i$ with  $m<i<m'$ is arbitrary.

Note with care that $\al_m=\tau(\al_m)$ in the diagram on the left is even.
A special case of odd $\al_m$ occurs if $m=1$, that is, $|\Pi|=|\bar \Pi_\l|=1$. It should be attributed to $\tau=\id$.
\subsubsection{Identical  $\tau$ and shaft spherical subalgebras}
\label{Sec_Shaft}
The remaining family of spherical subalgebras of the $\Ag$-series also plays a part in the structure of subalgebras in series $\Bg,\Cg,\Dg$.

Suppose that all connected components of $D_\l$ consist of one node. Such a node can be odd only if $|\Pi_\l|=1$, because $w_\l$ flips its basic weights.
This case has been already considered in the preceding  subsection.  Thus, if the number of components of $D_\l$ is 2 or higher, all roots in  $\Pi_\l$ are even, and $\tau=\id$.
Within the setting of \cite{AMS}, there were two families of such  diagrams. In full generality, they can be included in a uniform  description as follows.
\begin{itemize}
  \item All odd nodes are white.
  \item The even part of the diagram splits to an ordered sequence of connected components (possibly empty in the case
of two neighbouring odd roots). If an odd root occupies the leftmost (respectively rightmost position), we assume that the even connected component on
the left (respectively on the right) is  empty.
  \item
  Each even component  either consists of white nodes  or is an alternating diagram
$
 \begin{picture}(45,10)
\put(2,3){\circle*{3}}
\put(3,3){\line(1,0){5}}
\put(11,0){$\cdots$}
\put(26,3){\line(1,0){5}}
\put(32,3){\circle{3}}
\put(34,3){\line(1,0){7}}
\put(42,3){\circle*{3}}
  \end{picture}
$
of odd rank.
\item The types of even components (an empty component is regarded as white) must alternate.
\end{itemize}
This description results from application of Lemmas \ref{non-grad-sel-rule},  \ref{isolated odd},  and  \ref{le-b-d}.

Diagrams of this kind have the following three types of the rightmost block:
\be
\label{rightmost_block}
 \begin{picture}(45,10)
\put(2,3){\circle*{3}}
\put(3,3){\line(1,0){5}}
\put(11,0){$\cdots$}
\put(26,3){\line(1,0){5}}
\put(32,3){\circle{3}}
\put(34,3){\line(1,0){7}}
\put(42,3){\circle*{3}}
  \end{picture}
\quad
 \begin{picture}(56,10)
\put(2,3){\circle*{3}}
\put(3,3){\line(1,0){5}}
\put(11,0){$\cdots$}
\put(26,3){\line(1,0){5}}
\put(32,3){\circle{3}}
\put(34,3){\line(1,0){7}}
\put(42,3){\circle*{3}}
\put(43.5,3){\line(1,0){7}}
\put(51.5,1.5){\framebox(3,3)}
  \end{picture}
\quad
 \begin{picture}(45,10)
\put(2,3){\circle{3}}
\put(3,3){\line(1,0){5}}
\put(11,0){$\cdots$}
\put(26,3){\line(1,0){5}}
\put(32,3){\circle{3}}
 \end{picture}
\ee
We will use this fact when studying invariants of the spherical subalgebras of concern.
\begin{definition}
  A shaft spherical subalgebra $\k$ in a general linear Lie superalgebra $\g$ is the one that corresponds to a triple $(\g,\l,\tau=\id)$ described above.
  Its DDD is called shaft diagram.
\end{definition}
\noindent
Later on, we will also use this term  for spherical subalgebras contained in the shaft part $\s$ of  ortho-symplectic  $\g$.
\begin{remark}
\label{outstanding gl}
\em
  Two special  Satake diagrams of shape $\Ag$ with $\tau=\id$,
\be
\label{3NODES}
\begin{picture}(7,10)
\put(1.5,1.5){\framebox(3,3)}
\end{picture}
\quad \quad
\begin{picture}(60,15)
\put(0,3){\circle*{3}}
\put(1.5,3){\line(1,0){27}}

\put(28.5,1.5){\framebox(3,3)}

\put(31.5,3){\line(1,0){24}}
\put(57,3){\circle*{3}}
 \end{picture}
\ee
that fall outside the above classification give rise to proper spherical subalgebras in $\g\l(1|1)$ and $\g\l(2|2)$ of codimension 2 and 5, respectively.
They  cannot
be extended to  $\g\l$-diagrams of higher white  rank because of Lemmas \ref{isolated odd} and  \ref{le-b-d}.
Although the right DDD is topologically isomorphic to one of  $\o\s\p(4,2)$, they define different  spherical subalgebras.
\end{remark}
Shaft spherical subalgebras    play a role in construction of spherical subalgebras of type $\Bg,\Cg,\Dg$ in the subsequent sections.
Their quantum counterparts  are related to the so called twisted reflection equation of the $\Ag$-type, see e.g. \cite{AMS}.

\subsection{Orthosymplectic $\g$}
For ortho-symplectic $\g$ with the tail subalgebra $\t$, we separate three classes of Satake diagrams:
\begin{itemize}
  \item all nodes in $D_\t$ are  black,
  \item  all nodes in $D_\t$ are white,
  \item one node in $D_\t$ of shape $\Dg$  is black while the other is white.
\end{itemize}
We describe them in what follows.

Up to the triality automorphism of $\s\o(8)$, recall that $\tau=\id$ when restricted to $D_\s\subset D_\g$, and Satake diagrams of shapes $\Bg,\Cg,\Dg$ cannot have
isolated black odd nodes (such a  node cannot be a component of $D_\l$).
In Sections \ref{Sec_Black_tail}--\ref{Sec_Mixed_tail} we assume that $\l$ is regular; irregular $\l$ are excluded separately in Section \ref{Sec_irreg}.

A key ingredient of our approach is a reduction of $\k$  to its certain subalgebras that are
spherical with respect to certain Satake sub-diagrams.
One of them will be a shaft sub-diagram contained in $D_\s$.
It is an $\Ag$-shape diagram corresponding to the identical $\tau$.
The exceptional diagrams (\ref{3NODES}) cannot occur here: their extensions are excluded by Lemma \ref{le-b-d} in shapes $\Bg,\Cg$ and by Lemma \ref{sel-rul-d} in shape $\Dg$.
Therefore it has to be shaft.

\subsubsection{Black tail diagrams}
\label{Sec_Black_tail}
Black sub-diagram $D_\l$ that includes $D_\t$ may have an arbitrary admissible partition to even and odd  nodes. Let $D^\flat$ denote the sub-graph comprising the rightmost white node in $D_\g$, call it $\al$, and all black nodes to the right
including  $D_\t$ of arbitrary shape.
Our reduction rule asserts that the  complementary part $D\backslash D^\flat$ inherits a structure of a shaft Satake diagram.
Here is the list of  examples of small white rank:
\be
\label{Black-tail}
\begin{picture}(18,10)
\put(2,3){\circle{3}}
\put(4,3){\line(1,0){7}}
\put(10.5,-0.5){$\blacklozenge$}
\put(1,7){$\scriptscriptstyle \al$}
\end{picture}
 \quad
 \begin{picture}(18,10)
\put(1,7){$\scriptscriptstyle \al$}
\put(0.5,1.5){\framebox(3,3)}
\put(4,3){\line(1,0){7}}
\put(10.5,-0.5){$\blacklozenge$}
\end{picture}
 \quad
\begin{picture}(28,10)
\put(11,7){$\scriptscriptstyle \al$}
\put(2,3){\circle*{3}}
\put(4,3){\line(1,0){7}}
\put(12,3){\circle{3}}
\put(14,3){\line(1,0){7}}
\put(20.5,-0.5){$\blacklozenge$}
\end{picture}
 \quad
\begin{picture}(28,10)
\put(11,7){$\scriptscriptstyle \al$}
\put(2,3){\circle{3}}
\put(4,3){\line(1,0){7}}
\put(12,3){\circle{3}}
\put(14,3){\line(1,0){7}}
\put(20.5,-0.5){$\blacklozenge$}
\end{picture}
 \quad
\begin{picture}(38,10)
\put(22,7){$\scriptscriptstyle \al$}
\put(2,3){\circle*{3}}
\put(3,3){\line(1,0){7}}
\put(10.5,1.5){\framebox(3,3)}
\put(14,3){\line(1,0){7}}
\put(23,3){\circle{3}}
\put(24,3){\line(1,0){7}}
\put(30.5,-0.5){$\blacklozenge$}
\end{picture}
 \quad
\begin{picture}(28,10)
\put(11,7){$\scriptscriptstyle \al$}
\put(2,3){\circle*{3}}
\put(3,3){\line(1,0){7}}
\put(10.5,1.5){\framebox(3,3)}
\put(14,3){\line(1,0){7}}
\put(20.5,-0.5){$\blacklozenge$}
\end{picture}
 \quad
\begin{picture}(28,10)
\put(11,7){$\scriptscriptstyle \al$}
\put(2,3){\circle{3}}
\put(3,3){\line(1,0){7}}
\put(10.5,1.5){\framebox(3,3)}
\put(14,3){\line(1,0){7}}
\put(20.5,-0.5){$\blacklozenge$}
\end{picture}
 \quad
\begin{picture}(38,10)
\put(21,7){$\scriptscriptstyle \al$}
\put(2,3){\circle*{3}}
\put(3,3){\line(1,0){7}}
\put(10.5,1.5){\framebox(3,3)}
\put(14,3){\line(1,0){7}}
\put(21.5,1.5){\framebox(3,3)}
\put(25,3){\line(1,0){7}}
\put(31.5,-0.5){$\blacklozenge$}
\end{picture}
\ee
The big black rhombi on the right designate $D^\flat\backslash \{\al\}$.
The diagrams with non-empty shaft parts give smallest possible extensions of $D^\flat$.
 \subsubsection{White tail diagrams}
\label{Sec-White-tail}
 Removing $D_\t$ consisting of white nodes leaves  a shaft Satake sub-diagram  supported on $D_{\s}$ that is subject to selection rules when extended to $D_\g$.
Below  we list  diagrams with the smallest  $D_{\s}\not =\varnothing$ (which can be extended further to the left in accordance with
classification  of shaft Satake diagrams in Section \ref{Sec_Shaft}):
\begin{itemize}
  \item$\Bg$-shape tail:
$
\begin{picture}(15,10)
\put(2,3){\circle{3}}
\put(3,2){\line(1,0){7}}
\put(3,4){\line(1,0){7}}
\put(6.5,1){$\scriptstyle >$}
\put(13.5,3){\circle{3}}
\end{picture}
\quad
\begin{picture}(15,10)
\put(2,3){\circle*{3}}
\put(3,2){\line(1,0){7}}
\put(3,4){\line(1,0){7}}
\put(6.5,1){$\scriptstyle >$}
\put(13.5,3){\circle{3}}
\end{picture}
\quad
\begin{picture}(25,10)
\put(2,3){\circle*{3}}
\put(3,3){\line(1,0){7}}
\put(10.5,1.5){\framebox(3,3)}
\put(13,2){\line(1,0){7}}
\put(13,4){\line(1,0){7}}
\put(16.5,1){$\scriptstyle >$}
\put(23.5,3){\circle{3}}
\end{picture}
\quad
\begin{picture}(15,10)
\put(2,3){\circle{3}}
\put(3,2){\line(1,0){7}}
\put(3,4){\line(1,0){7}}
\put(6.5,1){$\scriptstyle >$}
\put(12,1.5){\framebox(3,3)}
\end{picture}
\quad
\begin{picture}(15,10)
\put(2,3){\circle*{3}}
\put(3,2){\line(1,0){7}}
\put(3,4){\line(1,0){7}}
\put(6.5,1){$\scriptstyle >$}
\put(12,1.5){\framebox(3,3)}
\end{picture}
\quad
\begin{picture}(28,10)
\put(2,3){\circle*{3}}
\put(3,3){\line(1,0){7}}
\put(10.5,1.5){\framebox(3,3)}
\put(13,2){\line(1,0){7}}
\put(13,4){\line(1,0){7}}
\put(16.5,1){$\scriptstyle >$}
\put(22,1.5){\framebox(3,3)}
\end{picture}
$
 \item $\Cg$-shape  tail:
$
\begin{picture}(15,10)
\put(1.5,3){\circle{3}}
\put(5.5,2){\line(1,0){7}}
\put(5.5,4){\line(1,0){7}}
\put(2.5,1){$\scriptstyle <$}
\put(13.5,3){\circle{3}}
\end{picture}
\quad
\begin{picture}(15,10)
\put(1.5,3){\circle*{3}}
\put(5.5,2){\line(1,0){7}}
\put(5.5,4){\line(1,0){7}}
\put(2.5,1){$\scriptstyle <$}
\put(13.5,3){\circle{3}}
\end{picture}
\quad
\begin{picture}(25,10)
\put(2,3){\circle*{3}}
\put(3,3){\line(1,0){7}}
\put(10,1.5){\framebox(3,3)}
\put(15.5,2){\line(1,0){7}}
\put(15.5,4){\line(1,0){7}}
\put(12.5,1){$\scriptstyle <$}

\put(23.5,3){\circle{3}}
\end{picture}
$

\item $\Dg$-shape, $\tau=\id$:
$
\quad
\begin{picture}(23,10)
 \put(1,3){\circle{3}}

\put(2.5,3.5){\line(1,1){10}}\put(2.5,3.2){\line(1,-1){10}}
\put(13,14){\circle{3}}\put(13,-8){\circle{3}}

\end{picture}
\quad
\begin{picture}(38,10)
\put(0,3){\circle*{3}}
 \put(1.5,3){\line(1,0){12}}

 \put(14,1.5){\framebox(3,3)}

\put(17.5,3.5){\line(1,1){10}}\put(17.5,3.2){\line(1,-1){10}}
\put(28,14){\circle{3}}\put(28,-8){\circle{3}}

\end{picture}
\quad
\begin{picture}(15,10)
\put(0,3){\circle*{3}}
\put(0,3.5){\line(1,1){11}}\put(0,3.5){\line(1,-1){11}}
\put(11,13){\framebox(3,3)}\put(11,-8){\framebox(3,3)}
\put(11,13){\line(0,-1){18}}\put(14,13){\line(0,-1){18}}
\end{picture}
$
\item $\Dg$-shape tail, $\tau\not =\id$:
$
\quad
\begin{picture}(22,20)
\put(0,3){\circle{3}}
 \put(1.5,3.5){\line(1,1){10}}\put(1.5,3.2){\line(1,-1){10}}
\put(12.5,14){\circle{3}}\put(12.5,-8){\circle{3}}
\qbezier(15,-7)(22,4)(15,14)
\put(16.5,12){\vector(-2,3){2}}\put(16.5,-5){\vector(-2,-3){2}}

\end{picture}
\quad
\begin{picture}(22,20)
\put(0,3){\circle*{3}}
    \put(1.5,3.5){\line(1,1){10}}\put(1.5,3.2){\line(1,-1){10}}
\put(12.5,14){\circle{3}}\put(12.5,-8){\circle{3}}
\qbezier(15,-7)(22,4)(15,14)
\put(16.5,12){\vector(-2,3){2}}\put(16.5,-5){\vector(-2,-3){2}}
\end{picture}
\quad
\begin{picture}(38,20)
\put(0,3){\circle*{3}}
 \put(1.5,3){\line(1,0){12}}

 \put(14,1.5){\framebox(3,3)}

\put(17.5,3.5){\line(1,1){10}}\put(17.5,3.2){\line(1,-1){10}}
\put(28,14){\circle{3}}\put(28,-8){\circle{3}}
\qbezier(31,-7)(38,4)(31,14)
\put(32.5,12){\vector(-2,3){2}}\put(32.5,-5){\vector(-2,-3){2}}
\end{picture}
\quad
\begin{picture}(25,20)
\put(0,3){\circle*{3}}
\put(0,3.5){\line(1,1){11}}\put(0,3.5){\line(1,-1){11}}
\put(11,13){\framebox(3,3)}\put(11,-8){\framebox(3,3)}
\put(11,13){\line(0,-1){18}}\put(14,13){\line(0,-1){18}}
\qbezier(16,-7)(23,4)(16,14)
\put(17.5,12){\vector(-2,3){2}}\put(17.5,-5){\vector(-2,-3){2}}
\end{picture}
\quad
\begin{picture}(25,20)
\put(-1.5,3){\circle{3}}
\put(0,3.5){\line(1,1){11}}\put(0,3.5){\line(1,-1){11}}
\put(11,13){\framebox(3,3)}\put(11,-8){\framebox(3,3)}
\put(11,13){\line(0,-1){18}}\put(14,13){\line(0,-1){18}}
\qbezier(16,-7)(23,4)(16,14)
\put(17.5,12){\vector(-2,3){2}}\put(17.5,-5){\vector(-2,-3){2}}
\end{picture}
\quad
\begin{picture}(35,20)
\put(0,3){\circle*{3}}
 \put(2,3){\line(1,0){12}}
\put(14,1.5){\framebox(3,3)}
\put(17,3.5){\line(1,1){11}}\put(17,3.5){\line(1,-1){11}}
\put(28,13){\framebox(3,3)}\put(28,-8){\framebox(3,3)}
\put(28,13){\line(0,-1){18}}\put(31,13){\line(0,-1){18}}
\qbezier(33,-7)(40,4)(33,14)
\put(34.5,12){\vector(-2,3){2}}\put(34.5,-5){\vector(-2,-3){2}}
\end{picture}
$
\end{itemize}
Other shaft extensions starting with
$
\quad
\begin{picture}(20,20)
 \put(1,3){\circle*{3}}
\put(2.5,3.5){\line(1,1){10}}\put(2.5,3.2){\line(1,-1){10}}
\put(13,14){\circle{3}}\put(13,-8){\circle{3}}
\end{picture}
\quad
\begin{picture}(20,20)
\put(-2,3){\circle{3}}
\put(0,3.5){\line(1,1){11}}\put(0,3.5){\line(1,-1){11}}
\put(11,13){\framebox(3,3)}\put(11,-8){\framebox(3,3)}
\put(11,13){\line(0,-1){18}}\put(14,13){\line(0,-1){18}}
\end{picture}
\quad
\begin{picture}(35,20)
\put(0,3){\circle*{3}}
 \put(2,3){\line(1,0){12}}
\put(14,1.5){\framebox(3,3)}
\put(17,3.5){\line(1,1){11}}\put(17,3.5){\line(1,-1){11}}
\put(28,13){\framebox(3,3)}\put(28,-8){\framebox(3,3)}
\put(28,13){\line(0,-1){18}}\put(31,13){\line(0,-1){18}}
\end{picture}
$
 are forbidden by Lemmas \ref{non-grad-sel-rule} and \ref{isolated odd}.
\subsubsection{Mixed coloured tail of $\Dg$-shape}
\label{Sec_Mixed_tail}
Suppose that one tail root, say, $\al_{n-1}$, is black and the other one, $\al_n$, is white.
\begin{lemma}
  The tail roots are even.
\end{lemma}
\begin{proof}
Suppose the opposite, that the tail roots are odd.
  Let $D_\mg$ be the connected component of $D_\l$ containing $\al_{n-1}$. The subalgebra $\mg\subset \l$ is of shape $\Ag$ and rank $m$.
  If the  polarization of $\mg$  is not symmetric,
  then $\l$  is irregular. Suppose it is symmetric (and therefore $m>1$), then $\tau$ must be an automorphism of the Dynkin diagram $D_\g$,
  see Remark \ref{even-isometry}.
   But that is impossible, because $\al_{n-1}$ displaced by $\tau$ is the only black node with a double link to a white node (that is $\al_n$).
\end{proof}
\noindent
Let us remind the reader  that irregular $\l$ are considered in Section \ref{Sec_irreg}.
\begin{propn}
  If $\al_{n-2}$ is black, then $\rk \>\g=4$, $\rk \>\l=3$, and the Satake diagram
  is isomorphic to black tailed with all even nodes.
\end{propn}
\begin{proof}
  Suppose that $\al_{n-2}$ is black. Then the $\l$-module $V_{\al_n}^+$ can be  self-dual only
  if $\al_{n-3}$ is also black.
  By the same reason, the sub-diagram  generated by $\{\al_{n-i}\}_{i=1}^3$ should be an even connected
   component in $D_\l$. The Satake diagram cannot be extended to the left because the $\l$-module $V_{\al_{n-4}}^+$ is not self-dual and has
   no isomorphic partner to mix with. Thus we conclude that $n=4$. The subalgebra $\l$ with odd black $\al_{n-3}$ and $\al_{n-2}$ is irregular, so the entire diagram has to be even.
\end{proof}
\noindent
Thus the only admissible diagram with even black $\al_{n-2}$ should be attributed to diagrams with black tail.

Suppose now that $\al=\al_{n-2}$  is white. We proceed similarly as we did  in Section \ref{Sec_Black_tail} with black tails.
The complement to $D^\flat=\{\al_{n-i}\}_{i=0}^2$ is a shaft Satake diagram, $S^l$, whose rightmost block is one of the three types displayed in  (\ref{rightmost_block}).
The following diagrams with minimal non-empty $S^l$ survive  the selection rules and can be extended further to the left:
\be
\label{mixed_tails_smallest}
\begin{picture}(50,10)
\put(0,3){\circle*{3}}
 \put(1,3){\line(1,0){12}}
\put(15,3){\circle{3}}
    \put(16.5,3.5){\line(1,1){10}}\put(16.5,3.2){\line(1,-1){10}}
\put(27.5,14){\circle{3}}\put(27.5,-8){\circle*{3}}
\put(32,13){$\scriptstyle \al_{n}$}
\put(32,-10){$\scriptstyle \al_{n-1}$}
\put(12,6){$\scriptstyle \al$}
\end{picture}
\quad
\begin{picture}(50,10)
\put(0,3){\circle{3}}
 \put(1,3){\line(1,0){12}}
 \put(13,1.5){\framebox(3,3)}
    \put(16.5,3.5){\line(1,1){10}}\put(16.5,3.2){\line(1,-1){10}}
\put(27.5,14){\circle{3}}\put(27.5,-8){\circle*{3}}
\put(32,13){$\scriptstyle \al_{n}$}
\put(32,-10){$\scriptstyle \al_{n-1}$}
\put(12,6){$\scriptstyle \al$}
\end{picture}
\quad
\begin{picture}(50,10)
\put(0,3){\circle*{3}}
 \put(1,3){\line(1,0){12}}
 \put(13,1.5){\framebox(3,3)}
\put(16,3){\line(1,0){12}}
 \put(28,1.5){\framebox(3,3)}
    \put(31.5,3.5){\line(1,1){10}}\put(31.5,3.2){\line(1,-1){10}}
\put(42.5,14){\circle{3}}\put(42.5,-8){\circle*{3}}
\put(47,13){$\scriptstyle \al_{n}$}
\put(47,-10){$\scriptstyle \al_{n-1}$}
\put(27,6){$\scriptstyle \al$}
\end{picture}
\ee
The other possible  DDD with even $\al$ include
$
\quad
\begin{picture}(50,20)
\put(0,3){\circle{3}}
 \put(1,3){\line(1,0){12}}
\put(15,3){\circle{3}}
    \put(16.5,3.5){\line(1,1){10}}\put(16.5,3.2){\line(1,-1){10}}
\put(27.5,14){\circle{3}}\put(27.5,-8){\circle*{3}}
\put(32,13){$\scriptstyle \al_{n}$}
\put(32,-10){$\scriptstyle \al_{n-1}$}
\put(12,6){$\scriptstyle \al$}
\end{picture}
$
and
$
\quad
\begin{picture}(60,20)
\put(0,3){\circle*{3}}
 \put(1,3){\line(1,0){12}}
 \put(13,1.5){\framebox(3,3)}
\put(16,3){\line(1,0){12}}
 \put(30,3){\circle{3}}
    \put(31.5,3.5){\line(1,1){10}}\put(31.5,3.2){\line(1,-1){10}}
\put(42.5,14){\circle{3}}\put(42.5,-8){\circle*{3}}
\put(47,13){$\scriptstyle \al_{n}$}
\put(47,-10){$\scriptstyle \al_{n-1}$}
\put(27,6){$\scriptstyle \al$}
\end{picture}
$.
They violate the selection rule of Lemma \ref{non-grad-sel-rule}, as well as the diagram
$
\quad
\begin{picture}(15,25)
 \put(1,3){\circle{3}}

\put(2.5,3.5){\line(1,1){10}}\put(2.5,3.2){\line(1,-1){10}}
\put(13,14){\circle{3}}\put(13,-8){\circle*{3}}

\end{picture}
\quad
$
with empty $S^l$.
 The diagrams with  odd $\al$ in (\ref{mixed_tails_smallest}) extend a valid Satake diagram
 $
\begin{picture}(50,25)
 \put(13,1.5){\framebox(3,3)}
    \put(16.5,3.5){\line(1,1){10}}\put(16.5,3.2){\line(1,-1){10}}
\put(27.5,14){\circle{3}}\put(27.5,-8){\circle*{3}}
\put(32,13){$\scriptstyle \al_{n}$}
\put(32,-10){$\scriptstyle \al_{n-1}$}
\put(12,6){$\scriptstyle \al$}
\end{picture}
$ with  $S^l=\varnothing$.
Its extension that includes
$
\begin{picture}(50,25)
\put(0,3){\circle*{3}}
 \put(1,3){\line(1,0){12}}
 \put(13,1.5){\framebox(3,3)}
    \put(16.5,3.5){\line(1,1){10}}\put(16.5,3.2){\line(1,-1){10}}
\put(27.5,14){\circle{3}}\put(27.5,-8){\circle*{3}}
\put(32,13){$\scriptstyle \al_{n}$}
\put(32,-10){$\scriptstyle \al_{n-1}$}
\put(12,6){$\scriptstyle \al$}
\end{picture}
$
is forbidden by Lemma \ref{sel-rul-d}.

\vspace{5pt}
Thus there are three series of Satake diagrams with tail roots of different colour.
\subsubsection{No Satake diagrams with irregular $\l$  }
\label{Sec_irreg}

By Proposition 2.6, an irregular connected component is of type $\Ag$ with non-symmetric polarization. Suppose first that it lies in the shaft, and let $\zt_k,\ldots,\zt_{k+r}$ be its basic weights. Applying the requirement $\deg\widetilde\mu=\deg\mu$ together with (\ref{1nd-cond})--(\ref{2nd-cond}) successively to the extremal weights gives
\[
(\zt_{k+j},\zt_{k+j})=(\zt_{k+r-j},\zt_{k+r-j}),\qquad j=0,\ldots,r.
\]
Hence the grading is symmetric, a contradiction. The same argument applies to every irregular component not meeting a terminal node. Thus only a $\Dg$-component containing exactly one terminal root remains.

Let $\g$ be an even ortho-symplectic Lie superalgebra of rank $n\geqslant 3$ and shape $\Dg$.
We start with the study of  the module $W=V^+_{\al_n}$ over the maximal $\g\l$-subalgebra in $\g'\subset\g$
with the root basis $\al_1,\ldots, \al_{n-1}$. Let $\s\subset \g'$ denote the subalgebra with the root basis $\al_1,\ldots, \al_{n-2}$.
The module $W$ contains an $\s$-submodule with weights $\zt_1+\zt_n, \ldots, \zt_{n-1}+\zt_n$.

We choose the grading on $V$ such that $\deg(v_n)=0$.
Applying to $e_{\zt_1+\zt_n}$ the subalgebra $\s$, we conclude,
that $W$ has a weight $\zt_1+\zt_2$. If $\deg(v_1)=\deg(v_n)$ and therefore $\eps_1=\eps_n=1$, then $2\zt_1\not \in \Rm_\g$. Otherwise
$\deg(v_1)$ is opposite to $\deg(v_n)$, $\eps_1=-1$, and  $e_{1,1'}\in W$ is the highest vector:
$$
e_{\zt_1-\zt_2}\tr e_{\zt_1+\zt_2}=e_{\zt_1-\zt_2} e_{\zt_1+\zt_2}-(-1)^{\bar 1+\bar 2} e_{\zt_1+\zt_2}e_{\zt_1-\zt_2}\propto (1-\eps_1)e_{1,1'}.
$$
We have demonstrated that
\begin{lemma}
  The vector $e_{2\zt_1}$ is the highest in $W$, provided $\deg(v_1)=1$. Otherwise the highest vector in $W$ is $e_{\zt_1+\zt_2}$.
\end{lemma}
\begin{corollary}
\label{irreg-s-dual}
  The module $W$ is self-dual as a $[\g',\g']$-module if and only if $\g=\s\o(8)$.
\end{corollary}
\begin{proof}
$W$ can be self-dual only if $\rk\>\g=4$.
  Pairing the roots of $\g'$ with the extremal weights we obtain a system of equalities
$$
(\zt_1-\zt_2,\zt_1+\zt_2+\zt_3+\zt_4)=(\zt_3-\zt_4,\zt_1+\zt_2+\zt_3+\zt_4)=(\zt_2-\zt_3,\zt_1+\zt_2+\zt_3+\zt_4)=0.
$$
That is,
$
(\zt_1,\zt_1)=\ldots =(\zt_4,\zt_4).
$
Therefore all basis vectors $v_i$ have the same degree. Conversely, for $\g=\s\o(8)$, one has $W\simeq\Lambda^2\C^4$, and the wedge pairing $W\otimes W\to\Lambda^4\C^4\simeq\C$ is non-degenerate and $\s\l(4)$-invariant.
\end{proof}
\begin{lemma}
\label{lem-basic-dual}
  The module $W$ is not dual to the basic  $\g'$-module.
\end{lemma}
\begin{proof}
  If the rank $n-1$ of $\g'$ is greater than 2, then $\dim W\geqslant \frac{n(n-1)}{2}> n$.
So we can assume that $n=3$ and $\eps_1=1$, so that $e_{\ve_1+\ve_2}$ would be the highest vector in $W$
(otherwise $\dim W$ would be $4>n$).

Let us realize the module $\C^n$ as a submodule in a larger-rank algebra where $\g$ is embedded as the tail part, and denote by $\zt_0$ the additional basic weight preceding $\zt_1$.
Then the weight $\zt_0-\zt_1+\zt_2+\zt_3$ should be orthogonal to $\zt_1-\zt_2$ and $\zt_2-\zt_3$. These requirements are
contradictory, as they force $(\zt_1,\zt_1)=-(\zt_2,\zt_2)=-(\zt_3,\zt_3)$. This implies $\eps_1=-1$, which is a contradiction.
\end{proof}
\begin{propn}

  There is no admissible super-symmetric triple $(\g,\l,\tau)$ with irregular $\l$.

\end{propn}
\begin{proof}
By the reduction above, it remains only to consider the $\Dg$-case with exactly one terminal root in the irregular component.
We can assume that  the tail roots of $\g$ have different colour and we can choose $\al_{n-1}\in \Pi_{\l}$.
We can also assume that $D_\l$ is connected and set $k=\rk\> \l$.
The only white roots that are linked to $D_\l$ are $\al_n$ and $\al_{n-k-1}$, provided $k+1<n$.

If $\rk \>\l=1$, then $\l$ is irregular only if $\al_n$ and $\al_{n-1}$ are odd. Then the only root that is connected to $\al_{n-1}$ is
$\al_{n-2}$, the modules $V^+_{\al_{n-2}}$ and $V^-_{\al_{n-2}}$ are not isomorphic and have no other isomorphic modules to mix with.

If $\rk \>\l>1$, the only case when $V^+_{\al_n}\simeq V^-_{\al_n}$  is with regular $\l$, thanks to Corollary \ref{irreg-s-dual}.
So we can assume that $V^+_{\al_n}\not \simeq V^-_{\al_n}$. The only candidate for $\tau(\al_n)\in \bar \Pi_\l$
is $\al_{n-k-1}$. By Lemma \ref{lem-basic-dual}, $V^-_{\al_{n-k-1}}$ is an $\l$-module  of dimension ${k+1}$,
while $\dim V^+_{\al_n}\geqslant \frac{(k+1)k}{2}$.
For $k>2$ the dimensions are different, while the remaining case $k=2$ is excluded by Lemma \ref{lem-basic-dual}.
\end{proof}
\begin{remark}
  \em
Some diagrams for minimal symmetric polarization of ortho-symplectic $\g$ were overlooked in \cite{AMS}.
Those are
$$
 
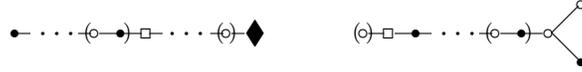
\begin{picture}(100,10)
\put(2,3){\circle*{3}}
\put(3,3){\line(1,0){5}}
\put(11,0){$\cdots$}
\put(26,3){\line(1,0){5}}
\put(28,1){$\scriptstyle($}
\put(32,3){\circle{3}}
\put(34,3){\line(1,0){7}}
\put(42,3){\circle*{3}}
\put(43,1){$\scriptstyle)$}
\put(43,3){\line(1,0){7}}
\put(50,1.5){\framebox(3,3)}
\put(53,3){\line(1,0){5}}
\put(60,0){$\cdots$}
\put(76,3){\line(1,0){5}}
\put(78,1){$\scriptstyle($}
\put(83,1){$\scriptstyle)$}
\put(82,3){\circle{3}}
\put(84,3){\line(1,0){5}}
\put(89,-0.5){$\blacklozenge$}

\end{picture}
\quad
\begin{picture}(120,10)

\put(18,1){$\scriptstyle($}
\put(23,1){$\scriptstyle)$}
\put(22,3){\circle{3}}
\put(24,3){\line(1,0){6}}
\put(30,1.5){\framebox(3,3)}

\put(34,3){\line(1,0){7}}
\put(42,3){\circle*{3}}
\put(43,3){\line(1,0){5}}
\put(51,0){$\cdots$}
\put(66,3){\line(1,0){5}}
\put(68,1){$\scriptstyle($}
\put(72,3){\circle{3}}
\put(74,3){\line(1,0){7}}
\put(82,3){\circle*{3}}
\put(83,1){$\scriptstyle)$}

 \put(83,3){\line(1,0){7}}
\put(92,3){\circle{3}}
    \put(93.5,3.5){\line(1,1){10}}\put(93.5,3.2){\line(1,-1){10}}
\put(104.5,14){\circle{3}}\put(104.5,-8){\circle*{3}}
\end{picture}
$$
and some diagrams of the  $\Dg$-shape with twisted white tail:
$$
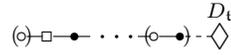
\quad
 \begin{picture}(105,10)

\put(18,1){$\scriptstyle($}
\put(23,1){$\scriptstyle)$}
\put(22,3){\circle{3}}
\put(24,3){\line(1,0){6}}
\put(30,1.5){\framebox(3,3)}

\put(34,3){\line(1,0){7}}
\put(42,3){\circle*{3}}
\put(43,3){\line(1,0){5}}
\put(51,0){$\cdots$}
\put(66,3){\line(1,0){5}}
\put(68,1){$\scriptstyle($}
\put(72,3){\circle{3}}
\put(74,3){\line(1,0){7}}
\put(82,3){\circle*{3}}
\put(83,1){$\scriptstyle)$}

\multiput(84,3)(4,0){3}{\line(1,0){2}}
\put(93,-0.5){$\lozenge$}
\put(92,10){$\scriptstyle D_\t$}
\end{picture}
$$
The parentheses mean periodicity with possible zero occurrence.
Without the block of white nodes on the left of the odd root, this turns to a diagram of Type I, in the classification of \cite{AMS}.

\end{remark}

\section{Non-trivial  super-symmetric pairs}
\label{Sec_Nontriviality}
In this section we  prove that Satake diagrams defined in the previous section allow for non-trivial super-symmetric pairs $(\g,\k)$.
We are going to  demonstrate that for certain $\vec c$ the subalgebra
$\k$ is smaller than $\g$. We set all parameters $\grave{c}_\al$ to zero, for the sake of simplicity.
We will study (even) matrix invariants of $\k$ and show that they are in a greater supply than those of $\g$.
Depending on a type of diagram, we will consider the following three actions
$$
x \tp A\mapsto \rho(x)A-A\rho(x), \quad  x \tp A\mapsto \rho(x)A+A\rho^t(x), \quad x \tp A\mapsto \rho(x)A-A\rho^\theta(x),
$$
for  $x\in \k$ and $A\in \End(V)$. Here $\theta$ is the conjugation with the flip operator $v_n \leftrightarrow v_{n'}$ when $\g\in\Dg$ of rank $n$.
The action on the left is adjoint, while the two others will be referred to as twisted.

Our method is reducing a  given Satake diagram to its smaller  sub-diagrams and finding invariants for the corresponding spherical subalgebras in $\k$.
A key role in this approach belongs to shaft spherical subalgebras because the shaft  constitutes  a bulk part of a general Dynkin graph of  shapes $\Bg,\Cg$, and $\Dg$.
We do that by induction on the white rank of the sub-diagrams.  Sub-diagrams of small white rank containing the tail are studied separately.
Finally we glue up the invariants of the shaft and tail parts of the diagram, to obtain an invariant of $\k$.
\subsection{Adjoint matrix $\k$-invariants for general linear  $\g$}
In this section we study  invariants of a spherical subalgebra $\k$ in general linear $\g$ corresponding to Satake diagrams (\ref{GL-I}) with $\tau\not =\id$.
We consider the adjoint action of $\k$ on $\End(V)$. Note that adjoint $\g$-invariants in $\End(V)$ are only  scalar matrices since $V$ is irreducible over $\g$.

Let $N-2m-1\geqslant 0$ be the rank of $\l\simeq\s\l(N-2m)$. Its natural extension $\hat\l=\g\l(N-2m)$ acts on the graded
vector space $\C^{N-2m}=\Span\{v_{m+1},\ldots,v_{N-m}\}$.
Consider a matrix
\be
A=(\mu+\la)\sum_{i=1}^{m}e_{ii}+\la \sum_{i=m+1}^{N-m}e_{ii}+\sum_{i=1}^{m}\bigl(a_i e_{i,i'}+a_{i'}e_{i',i}\bigr), \quad a_i a_{i'}=-\la\mu,
\label{GL-inv}
\ee
where $\la$ and $\mu$ are non-zero complex numbers (the eigenvalues).
We argue that $A$ is $\k$-invariant for certain values of mixture parameters.
Indeed, it  is known to be a $K$-matrix for a certain coideal subalgebra  $U_q(\k)$ with the  minimal symmetric grading on $V$
when $\l\subset \k$ is $\s\l(N-2m)$. cf.  \cite{AMS}. Since $A$ is independent of $q$, it is $\k$-invariant in the limit $q=1$.
It will stay invariant if we replace $\hat \l=\g\l(N-2m)$ with an arbitrary graded $\g\l= \End(\C^{N-2m})$, because
$A$ is a scalar on $\C^{N-2m}$ (the term proportional to $\la$ in (\ref{GL-inv})).
Furthermore, notice that we can choose simple root vectors and mixture parameters such that
the mixed generators $x_\al$ with $\tau(\al)\not =\al$ will be given by the same formulas independent of the grading:
$$
x_k=e_{k,k+1}+c_k e_{k',(k+1)'},\quad k<m, \quad m'\leqslant k,
$$
$$
x_m=e_{m,m+1}+c_m e_{m',m+1}, \quad x_{m'-1}=e_{(m+1)',m'}+c_{m'-1} e_{(m+1)',m}, \quad \tau(\al)\not =\al, \quad \forall \al\in \bar \Pi_\l.
$$
Let us demonstrate  that  generators $y_\al=h_{\al}-h_{\tilde \al}=h_\al+h_{\theta(\al)}\in \k$ with $\al\in \bar \Pi_\l$  commute with $A$
regardless of the grading.
First of all, observe that the linear mapping $\h^*\ni \la \mapsto h_\la\in \h$ acts on the basic weights by the assignment
$$
\zt_i\mapsto (-1)^{\bar i} e_{i,i}, \quad i=1, \ldots, N.
$$
The operator $\theta$ is even and takes $\zt_i$ to $\zt_{i'}$ for $i\leqslant m$. Due to  the equality $\bar i= \bar i'$ we can
write
$$
h_{\al_i}-h_{\tilde \al_i}=(-1)^{\bar i}( e_{i,i}+e_{i',i'})-(i\mapsto i+1) , \quad i<m,
$$
$$
h_{\al_m}-h_{\tilde \al_m}=(-1)^{\bar m}( e_{m,m}+e_{m',m'}) \mod \l.
$$
This way we reduce verification to calculations in the 2-dimensional subspaces $\Span(v_i,v_{i'})$ for each $i=1,\ldots, m$, where it becomes obvious.

When $\tau(\al_m)=\al_m$ in the present section, $\al_m$ is even by Section 3.1.1. We choose the corresponding mixed generator as
$$
x_m=(-1)^{\bar m+\bar m'}e_{m,m'}+c_m^2 e_{m',m},   \quad \tau(\al_m)=\al_m.
$$
Since $\tau(\al_m)=\al_m$, the generator $y_{\al_m}$ is zero. The exceptional odd rank-one case has $\tau=\id$ and is considered in Section \ref{Sec_twisted_GL}.

Let us compute the relation between the parameters of the matrix $A$ and the mixture parameters of the algebra $\k$.
\begin{propn}
  The matrix (\ref{GL-inv}) is $\k$-invariant provided $c_i=c_{N-i}$, for $i=1,\ldots, m-1$, and
$$
a_{k+1}=a_k c_k, \quad a_{(k+1)'}=a_{k'}/c_k, \quad k=1,\ldots, m-1,
$$
$$
\frac{a_{m'}}{a_m}=c_m c_{m'-1},
\quad
-\mu
=c_m a_m
,
\quad
\la
= \frac{a_{m'}}{c_m}, \quad \tau(\al_m)\not =\al_m,
$$
$$
\frac{a_{m'}}{a_m}=(-1)^{\bar m+\bar m'} c_m^2,\quad \la=-\mu=\sqrt{(-1)^{\bar m+\bar m'}}\>\frac{a_{m'}}{c_m}, \quad \tau(\al_m)=\al_m \in \bar \Pi_\l.
$$
\end{propn}
\begin{proof}
  Direct calculation.
\end{proof}

\begin{corollary}
A  pair $(\g,\k)$ with Satake diagram (\ref{GL-I})  is non-trivial for any  $\vec c$ subject to $c_\al=c_{\tau(\al)}$.
\end{corollary}

As a byproduct of the above analysis, we explain an interesting fact  observed in \cite{AlgMS1} that $K$-matrix of type $\Ag$ is independent of grading
on $V$.
The only requirement is that $K$ should be even, which is fulfilled for (\ref{GL-inv}).  Indeed,
the grading relative to the right diagram in (\ref{GL-I}) is symmetric on the subspace $\Span\{v_i,v_{i'}\}_{i=1}^m$. Furthermore,
the restriction of $K$ to $\Span\{v_i\}_{m<i<m'}$ is a scalar matrix, which  is even relative to any grading on this subspace,
thereby allowing for arbitrary polarization of  $\l$.

\subsection{Twisted matrix invariants of shaft $\k$}
\label{Sec_twisted_GL}
In this section, $\g$ is  general linear. We study  Satake diagrams with $\tau=\id$.
Let us first consider the outstanding diagrams (\ref{3NODES}).

The case of $\rk \> \g=1$ depends on the parity of the only root $\al$.
The mixed operators
$
x_\al=e_{12}+c_\al e_{21}
$ can be fixed independent of grading.

For even $\al$, the matrix $e_{11}-c_\al e_{22}$ is a non-zero twisted $\k$-invariant. For odd $\al$, the odd matrix $e_{12}-c_\al e_{21}$ is adjoint $\k$-invariant. Thus $\k$ is proper in both cases.

For the right diagram in (\ref{3NODES}), let $x_2=e_{23}+c_2e_{41}$. The $\l\simeq\mathfrak{sl}_2\oplus\mathfrak{sl}_2$-submodule generated by $x_2$ is four-dimensional and odd, while its self-brackets add one central even direction. Hence $\dim\k=3+3+4+1=11<16=\dim\g\l(2|2)$, so this spherical subalgebra is proper.

Next we focus on invariants of the shaft subalgebras.
With every white root $\al\in \bar \Pi_\l$ we relate a mixed generator
$$
x_\al=e_\al+(-1)^{\deg(x_\al)} c_\al f_{\tilde \al}.
$$
 The root vector $f_{\tilde \al}$ is normalized as follows.
If $\bt$ is the only (even) black root linked to $\al$, then we set $f_{\tilde \al}=[f_{\bt},f_\al]$. If  $\al$ is connected
to a pair of  black roots $\bt_i$, $i=1,2$, then we take $f_{\tilde \al}=[f_{\bt_2},[f_{\bt_1},f_\al]]$.

Let $V$ be a graded vector space of the basic module for $\g$. We choose simple root vectors of $\g$ as
$$
f_{\al_i}\mapsto e_{i+1,i}, \quad  e_{\al_i}\mapsto (-1)^{\bar i} e_{i,i+1}.
$$
Let us enumerate white roots using the upper index in $\al^i$, where it varies from $1$ to the white rank $\ell=|\bar \Pi_\l|$, reverting the initial order on $\Pi$.
In this section we find  $\k$-invariants of the action
$$
\rho(x)A+(-1)^{\deg(A)\deg( x)} A\rho^t(x), \quad x\in \g, \quad A\in \End(V),
$$
where $t$ is the matrix super-transposition $A^t_{ij}=(-1)^{(\bar i+\bar j)\bar j}A_{ji}$.
We are concerned only with even $A$, for which the invariance  condition  simplifies to
$$
\rho(x)A+ A\rho^t(x)=0, \quad x\in \k.
$$
Note that $\g$-invariants of this kind reduce to the zero matrix only.

Consider an even  block-diagonal matrix
\be
\label{twisted-block}
\mathcal{I}_{\ell}=a(a_{\ell+1}\si_{\ell+1}+\ldots +a_2\si_2+ a_1 \si_1),
\ee
where $a_1=1$ and, writing $\al^i=\zt_{r_i}-\zt_{r_i+1}$,

\[
a_{i+1}=-(-1)^{\bar r_i\,\overline{r_i+1}}\,c_{\al^i}^{-1}a_i,
\qquad
a_{i+1}=\prod_{k=1}^{i}\left[-(-1)^{\bar r_k\,\overline{r_k+1}}c_{\al^k}^{-1}\right],
\]
and the block $\si_i$ is of size 1 or
$
\si_i=\left(
\begin{array}{ccc}
  0 & 1 \\
  -1 & 0
\end{array}
\right)
$.
They are consecutively described from the bottom to top by the requirement
that the upper block is of  size 2 if the leftmost node in the diagram is black; otherwise it is of size 1.
\begin{propn}
  The matrix $\mathcal{I}_\ell$ is $\k$-invariant.
\label{shaft_invariants}
\end{propn}
\begin{proof}
We do induction on the white rank  $\ell$ of the  diagram $S$. We take for the induction base the  diagrams of rank $\leqslant 3$ with
$\ell\leqslant 2$:
$$
\begin{picture}(4,10)
\put(1,3){\circle{3}}
\end{picture},
\quad
\begin{picture}(4,10)
\put(1,3){\circle*{3}}
\end{picture},
\quad
\begin{picture}(15,10)
\put(0.5,1.5){\framebox(3,3)}
\put(4,3){\line(1,0){7}}
\put(12,3){\circle*{3}}
\end{picture}
,
\quad
\begin{picture}(15,10)
\put(2,3){\circle{3}}
\put(4,3){\line(1,0){7}}
\put(12,3){\circle{3}}
\end{picture}
,
\quad
\begin{picture}(15,10)
\put(2,3){\circle*{3}}
\put(10.5,1.5){\framebox(3,3)}
\put(3,3){\line(1,0){7}}
\end{picture}
,\quad
\begin{picture}(25,10)
\put(2,3){\circle*{3}}
\put(4,3){\line(1,0){7}}
\put(12,3){\circle{3}}
\put(14,3){\line(1,0){7}}
\put(22,3){\circle*{3}}
\end{picture}
,\quad
\begin{picture}(25,10)
\put(2,3){\circle*{3}}
\put(3,3){\line(1,0){7}}
\put(10.5,1.5){\framebox(3,3)}
\put(14,3){\line(1,0){7}}
\put(22,3){\circle{3}}
\end{picture}
,\quad
\begin{picture}(25,10)
\put(2,3){\circle{3}}
\put(3,3){\line(1,0){7}}
\put(10.5,1.5){\framebox(3,3)}
\put(14,3){\line(1,0){7}}
\put(22,3){\circle*{3}}
\end{picture}
$$
For them the statement is checked directly.

Suppose that $\ell>2$.  Remove the leftmost white node $\al$ together with the part of $S$ on the left  of it (two nodes at most altogether).
The remaining part $S^r$ on the right is a Satake sub-diagram, for which the statement is true by induction assumption. Let $\k^r$
denote the shaft spherical subalgebra determined by $S^r$:
$$
\begin{picture}(150,20)
\put(20,0){$S=$}
\put(70,5){\oval(40,20)}\put(105,5){\oval(50,16)}
\put(54,1){$S^l$}\put(69,2){$\al $}\put(107,1){$S^r$}
\end{picture}
$$
The initial diagram is restored by appending the removed sub-graph to $S^r$. This case reduces to an already considered
situation with $\ell\leqslant 2$ if we ignore the nodes from $S^r$ that are not in the Satake
subdiagram $S^l=D(\al)$. The latter can be  one of
$$
\begin{picture}(15,10)
\put(0,7){$\scriptstyle \al$}
\put(2,3){\circle{3}}
\end{picture}
\quad
\begin{picture}(25,10)
\put(10,7){$\scriptstyle \al$}
\put(2,3){\circle*{3}}
\put(3,3){\line(1,0){7}}
\put(12,3){\circle{3}}
\put(14,3){\line(1,0){7}}
\put(22,3){\circle*{3}}
\end{picture}
\quad
 \begin{picture}(15,10)
\put(0.5,1.5){\framebox(3,3)}
\put(4,3){\line(1,0){7}}
\put(12,3){\circle*{3}}
\put(0,7){$\scriptstyle \al$}
\end{picture}
$$
Denote by $\k^l$ the spherical sub-diagram determined by this restriction. Let $m$ be the white rank of $S^l$.
The $\k^l$-invariant matrix $\mathcal{I}_m$ has exactly $m+1$ diagonal blocks. It can be scaled so that its first (lowest) block coincides with the  $\ell$-th (upper) block
in the $\k^r$-invariant matrix $\mathcal{I}_{\ell-1}$ and is therefore $\k$-invariant. The  block with number $m+1$ (upper) block in $\mathcal{I}_2$ is $\k^r$-invariant,
while the blocks with numbers $i< \ell$ in $\mathcal{I}_{\ell-1}$ are $\k^l$-invariant.
The  extension $\mathcal{I}_{\ell}$  of $\mathcal{I}_{\ell-1}$  by the $2$-nd and $m+1$-st blocks of  $\mathcal{I}_m$  is  invariant with respect to the entire $\k$.
\end{proof}
\noindent
Thus we have proved that shaft spherical subalgebras are non-trivial for each vector of non-zero mixture parameters.
\subsection{Transposition of shaft spherical subalgebras}
In this section we explore how transposition affects the shaft spherical subalgebra
$\k=\k(\vec c)$, that corresponds to  a Satake diagram of shape $\Ag$ with $\tau=\id$.
We will use these results  to construct invariants for spherical subalgebras of types $\Bg$, $\Cg$, and $\Dg$.
\begin{propn}
\label{transpose k}
  The matrix super-transposition takes $\k(\vec c)$ to
  the subalgebra $\k(\vec {c'}$), where $c'_\al=\pm c_\al^{-1}$ and the
  sign depends on the type of $D(\al)$. In particular, all signs equal  $+1$ if the grading is zero on the support
  modules of black $\s\l(2)$-subalgebras.
\end{propn}
\begin{proof}
  Clearly the transposition preserves $\l\subset \k$ (the Levi core of $\k$ is a direct sum of $\s\l(2)$).
The transformations of the mixture parameters of $\k$ are studied next. There are three possible cases
depending on the sub-diagram $D(\al)\subset S$, $\al \in \bar \Pi_\l$.

Assume first that $\al\in \bar \Pi_\l$ is isolated and even: $D(\al)=
\begin{picture}(4,10)
\put(1,3){\circle{3}}
\end{picture}.
$
Applying the transposition to the mixed generator, $x_\al=e_\al+c_\al f_\al=x_\al(c_\al)$, we get
$$
x_\al^t=f_\al+c_\al e_{\al}=c_\al(e_{\al}+c_\al^{-1}f_\al)\propto x_\al(c_\al^{-1}).
$$
Now suppose that $D(\al)=\begin{picture}(25,10)
\put(2,3){\circle*{3}}
\put(4,3){\line(1,0){7}}
\put(12,3){\circle{3}}
\put(14,3){\line(1,0){7}}
\put(22,3){\circle*{3}}
\end{picture}$.
Let $\bt,\gm$ be the black even roots linked with $\al$.
The mixed generator
$
x_\al(c_\al)=e_\al+c_\al [f_\gm,[f_\bt,f_\al]]
$
is the lowest vector in an $\l$-module it generates.
 Applying transposition
we get the highest vector
$$
 x_\al^t=f_\al+c_\al [[e_\al,e_\bt],e_\gm].
$$
We return back to the lowest  vector $[f_\gm,[f_\bt,x_\al^t]]$:
$$
[f_\gm,[f_\bt,f_\al]]+c_\al [[e_\al,h_\bt],h_\gm]=[f_\gm,[f_\bt,f_\al]]+c_\al (\al,\bt)(\al,\gm)e_\al\propto x_\al\bigl((\al,\bt)(\al,\gm)c_\al^{-1}\bigr)
=x_\al\bigl(c_\al^{-1}\bigr),
$$
because $(\al,\bt)=(\al,\gm)=\pm 1$.
Finally, suppose that $\al$ is odd and $D(\al)\simeq
\begin{picture}(15,10)
\put(2,3){\circle*{3}}
\put(10.5,1.5){\framebox(3,3)}
\put(3,3){\line(1,0){7}}
\end{picture}
$, where the black root is $\bt$. The generator $x_\al=x_\al(c_\al)=e_\al-c_\al [f_\bt,f_\al]$ is transposed to
$$
x_\al^t=\pm( f_\al+ c_\al [e_\al,e_\bt]),
$$
where the signs depend on the grading.
Returning to the lowest vector in this $\l$-submodule, we find
$$
[f_\bt,x_\al^t]\propto ([f_\bt,f_\al]+c_\al [h_\bt,e_\al])\propto [f_\bt,f_\al]+c_\al (\al,\bt)e_\al\propto x_\al\bigl(-(\al,\bt)c_\al^{-1}\bigr)= x_\al\bigl(c_\al^{-1}\bigr),
$$
provided that $(\al, \bt)=-1$. This condition will be fulfilled if we set the grading to zero  within the support of any
even component (and hence all) with black node.

The proof is complete, because
$D(\al)\simeq\begin{picture}(15,10)
\put(2,3){\circle*{3}}
\put(4,3){\line(1,0){7}}
\put(12,3){\circle{3}}
\end{picture}$
does not appear in a shaft diagram $S$ (it violates the selection rules).
\end{proof}

\subsection{Shaft $\k$-invariants via embedding $\Ag\subset \Bg,\Cg,\Dg$}
Satake diagrams for ortho-symplectic  $\g$ can be obtained by gluing up their  tail  and shaft
parts on a small intersection. It is therefore natural to construct invariants of their spherical subalgebras
out of invariants of smaller  subalgebras. To that end, we need to study  matrix invariants of a shaft
subalgebra via its embedding in an ortho-symplectic Lie superalgebra (under the adjoint action rather than the twisted action considered in the
previous section).

Let $\g$ be a general linear superalgebra, and $V\simeq \C^N$ its basic graded module with the natural basis vectors $v_i$ of weight $\zt_i$ and grade $\deg(v_i)=\bar i$,
$i=1,\ldots,N$.
Consider a graded vector space  $W=\Span\{w_{i'}\}_{i=1}^N$, where $i'=N+1-i$, and the basis vectors $w_{i'}$ of degree $\bar i'=\bar i$ carry weights $-\zt_i$.
Let $C\colon W\to V$ and $S\colon V\to W$ denote even  linear mappings defined by   $C(w_{i'})=v_i$,   $S(v_i)=\eps_i w_{i'}$,
where $\eps_i=\eps (-1)^{\bar i}$, $\eps =\pm 1$.
They satisfy $S=\eps C^{-1}\mathbb{P}$, where $\mathbb{P}=\sum_{k=1}^N (-1)^{\bar k}e_{k,k}$ is the parity operator.

Let $\rho\colon \g\to \End(V)$ denote  the natural representation homomorphism. Define a representation $\g\to \End(V)\op \End(W)$, $x\mapsto \rho(x)\op \tilde \rho(x)$ that preserves the
operator $(v,w)\mapsto \bigl(C(w),S(v)\bigr)$:
$$
\tilde \rho(x)=-S\rho^t(x)S^{-1}, \quad \rho(x)=-C\tilde \rho^t(x)C^{-1}, \quad \forall x\in \g.
$$
These two  equations follow from each other.

Consider the restriction of the adjoint representation of $\g$ on $\End(V\oplus W)$ to the subalgebra $\k$.
Denote by $\tilde \k$ the subalgebra $S\k^t S^{-1}$.
Let $A$ be the twisted $\k$-invariant as in Proposition \ref{shaft_invariants}, and let $\tilde A$ be a twisted $\tilde \k$-invariant.
According to Proposition \ref{transpose k}, it is obtained from $A$ by changing the mixture parameters $c_\al\mapsto c'_\al=\pm c_\al^{-1}$, $\al \in \bar \Pi_\l$,
and by conjugation with $S$, which flips the order of the diagonal blocks and multiplies each $2\times 2$-block by $-1$.
\begin{propn}
  There exists an embedding of the subspace  $\End(V)^\k$ of twisted  $\k$-invariants  into the space of  adjoint invariants $\End(V\oplus W)^\k$.
It is given by the assignment
$$
A\mapsto \left(
\begin{array}{ccc}
  0 & AS^{-1} \\
  \tilde A C^{-1}& 0
\end{array}
\right).
$$
\end{propn}
\begin{proof}
Permuting the matrices
$$
\left(
\begin{array}{ccc}
  \rho(x) & 0 \\
  0 & \tilde \rho(x)
\end{array}
\right)
\left(
\begin{array}{ccc}
  0 & AS^{-1} \\
  \tilde A C^{-1} & 0
\end{array}
\right)
=
\left(
\begin{array}{ccc}
  0 & AS^{-1} \\
  \tilde A C^{-1}& 0
\end{array}
\right)
\left(
\begin{array}{ccc}
  \rho(x) & 0 \\
  0 & \tilde \rho(x)
\end{array}
\right)
$$
where $x\in \k$, we arrive at the following equalities
$$
\rho(x) AS^{-1}=AS^{-1}\tilde \rho(x),\quad \tilde \rho(x)\tilde A C^{-1}=\tilde A C^{-1} \rho(x),
$$
$$
\rho(x) A =-A \rho^t(x),\quad \tilde \rho(x)\tilde A=-\tilde A \tilde \rho^t(x).
$$
They are satisfied by construction for all  $x\in \k$.
This completes the proof.
\end{proof}
\begin{corollary}
\nn
\label{4-param-inv}
    For each vector $\vec c$ of non-zero mixture parameters, the algebra $\k(\vec c)$ has a four-parameter invariant
$$
\left(\begin{array}{ccc}
  \mu & AS^{-1} \\
  \tilde A C^{-1}& \nu
\end{array}
\right)
\in \End(V\oplus W).
$$
\end{corollary}
\noindent
This result will be used for construction of $\k$-invariants for ortho-symplectic super-symmetric pairs.
Remark that different choices of root vectors (equivalently the representation $\rho$) can be compensated by different mixture parameters.
\subsection{Adjoint $\k$-invariants of black-tailed diagrams}
We argue that spherical  subalgebras defined via Satake diagrams  listed in (\ref{Black-tail}) have non-trivial invariants
whose structure we are going to describe.
Let us start with diagrams of white rank 1:
\be
\nn
\label{Black-tail-reduced}
\begin{picture}(18,10)
\put(2,3){\circle{3}}
\put(4,3){\line(1,0){7}}
\put(10.5,-0.5){$\blacklozenge$}
\put(1,7){$\scriptscriptstyle \al$}
\end{picture}
 \quad\quad
 \begin{picture}(18,10)
\put(1,7){$\scriptscriptstyle \al$}
\put(0.5,1.5){\framebox(3,3)}
\put(4,3){\line(1,0){7}}
\put(10.5,-0.5){$\blacklozenge$}
\end{picture}
 \quad\quad
\begin{picture}(28,10)
\put(11,7){$\scriptscriptstyle \al$}
\put(2,3){\circle*{3}}
\put(4,3){\line(1,0){7}}
\put(12,3){\circle{3}}
\put(14,3){\line(1,0){7}}
\put(20.5,-0.5){$\blacklozenge$}
\end{picture}
 \quad\quad
\begin{picture}(28,10)
\put(11,7){$\scriptscriptstyle \al$}
\put(2,3){\circle*{3}}
\put(3,3){\line(1,0){7}}
\put(10.5,1.5){\framebox(3,3)}
\put(14,3){\line(1,0){7}}
\put(20.5,-0.5){$\blacklozenge$}
\end{picture}
\ee
If the black sub-graph $D_\l$ is even, this  is a special case of  minimal symmetric polarization. Such $\k$
are quantized to coideal subalgebras $U_q(\k)$ which allow for $U_q(\k)$-fixed K-matrices, see \cite{AMS}, Theorem 3.1.
In the classical limit, they go to
$$
\mathcal{I}_1=(\mu+\la)\sum_{i=1}^{m}e_{ii}+a\si+a'\si^t+\la\sum_{i=m+1}^{m'-1}e_{ii}
$$
for certain $\la,\mu\in \C^\times $ and $a a'=-\mu\la$.
Here $m=1$ and $\si=e_{1,1'}$ if the leftmost node is white and $m=2$, $\si=e_{1,2'}-e_{2,1'}$ otherwise.

 The  matrix $\mathcal{I}_1$
will be $\k$-invariant if we allow for  arbitrary $\l$. Indeed, it will be
clearly $\l$-invariant, and the only mixed generator of $\k$ can be fixed independent of the polarization of $\l$.

Non-trivial $\k$-invariants for the remaining diagrams
$$
\begin{picture}(28,10)
\put(11,7){$\scriptscriptstyle \al$}
\put(2,3){\circle{3}}
\put(4,3){\line(1,0){7}}
\put(12,3){\circle{3}}
\put(14,3){\line(1,0){7}}
\put(20.5,-0.5){$\blacklozenge$}
\end{picture}
 \quad\quad
\begin{picture}(38,10)
\put(22,7){$\scriptscriptstyle \al$}
\put(2,3){\circle*{3}}
\put(3,3){\line(1,0){7}}
\put(10.5,1.5){\framebox(3,3)}
\put(14,3){\line(1,0){7}}
\put(23,3){\circle{3}}
\put(24,3){\line(1,0){7}}
\put(30.5,-0.5){$\blacklozenge$}
\end{picture}
 \quad
 \quad
\begin{picture}(28,10)
\put(11,7){$\scriptscriptstyle \al$}
\put(2,3){\circle{3}}
\put(3,3){\line(1,0){7}}
\put(10.5,1.5){\framebox(3,3)}
\put(14,3){\line(1,0){7}}
\put(20.5,-0.5){$\blacklozenge$}
\end{picture}
 \quad\quad
\begin{picture}(38,10)
\put(21,7){$\scriptscriptstyle \al$}
\put(2,3){\circle*{3}}
\put(3,3){\line(1,0){7}}
\put(10.5,1.5){\framebox(3,3)}
\put(14,3){\line(1,0){7}}
\put(21.5,1.5){\framebox(3,3)}
\put(25,3){\line(1,0){7}}
\put(31.5,-0.5){$\blacklozenge$}
\end{picture}
$$
from the list (\ref{Black-tail}) can be obtained from $\mathcal{I}_1$ by extending it to shaft $\k$-invariants
of diagrams
$
\begin{picture}(5,10)
\put(2,3){\circle{3}}
\end{picture}
$
and
$
\begin{picture}(15,10)
\put(2,3){\circle*{3}}
\put(3,3){\line(1,0){7}}
\put(10.5,1.5){\framebox(3,3)}
\end{picture}
$
using Corollary \ref{4-param-inv}. That is possible because only the first three terms in $\mathcal{I}_1$ matter for such subalgebras.

Let us express  $\Ic_1$ through the mixture parameter $c_\al$.
We normalize the mixed generator of the sub-diagram
$
  \begin{picture}(18,10)
\put(1,8){$\scriptscriptstyle \al$}
\put(0,1){$\scriptstyle\lozenge$}
\put(4,3){\line(1,0){7}}
\put(10.5,-0.5){$\blacklozenge$}
\end{picture}
$
as
$$
x_\al=-(-1)^{\overline{2}(\bar 1+\overline{2})} e_{1,{2}}+ e_{2',1'}-\eps_{2}(-1)^{\bar 1(\bar 1+\overline{2})}c_\al e_{2',1}+c_\al e_{1',2}.
$$
We find the following expressions for the parameters of the matrix $\Ic_1$:
$$
\mu=-\eps_1\la
,\quad
 a=-\eps_{2}(-1)^{\bar 1(\bar 1+\overline{2})}\la c_\al^{-1}
\quad
a' =- (-1)^{\overline{2}(\bar 1+\overline{2})}\la c_\al.
$$
For the diagram
$
\begin{picture}(28,10)
\put(11,8){$\scriptscriptstyle \al$}
\put(3,3){\circle*{3}}
\put(4,3){\line(1,0){7}}
\put(10,1){$\scriptstyle\lozenge$}
\put(14,3){\line(1,0){7}}
\put(20.5,-0.5){$\blacklozenge$}
\end{picture}
$
we enumerate the basis of the natural representation from $0$ to $0'$, with the conditions $(-1)^{\bar 0}=(-1)^{\bar 1}=-(-1)^{\bar 2}$ on the parities
and
$\eps_0=\eps_1=-\eps_2$ on
signatures of the invariant  form $C$.
If we fix the mixed generator as
$$
x_\al =
-(-1)^{\bar 2 (\bar 1+\bar 2)} e_{1,2}+ e_{2',1'}
-c_\al\eps_{2}(-1)^{\bar 1(\bar 1+\bar 2)} e_{2',0}+c_\al e_{0',2},
$$
then the parameters of $\Ic_1$ satisfy
$$
\mu =\la\eps_{1}, \quad a=\eps_{1}(-1)^{\bar 1(\bar 1+\bar 2)}\la c_\al^{-1}
,\quad
a'=(-1)^{\bar 2 (\bar 1+\bar 2)}\la c_\al.
$$

\subsection{Matrix invariants of ortho-symplectic spherical subalgebras}
In this final section we sketch a proof that all Satake diagrams from the $\Bg,\Cg,\Dg$-series classified in the previous section are non-trivial.
That is, each diagram corresponds to a proper spherical subalgebra for at least one  vector of mixture parameters.
In fact, that is true for all $\vec c\not=0$ if $\tau(\al)= \al$. If $\tau(\al) \not= \al$, then  the non-zero mixture parameters $c_\al$
and $c_{\tau(\al)}$ are related to each other with only one exception for the sub-diagram
$
\quad
\begin{picture}(40,10)
\put(0,3){\circle*{3}}
 \put(1,3){\line(1,0){12}}
\put(13,1.5){\framebox(3,3)}
\put(16,3.5){\line(1,1){11}}\put(16,3.5){\line(1,-1){11}}
\put(27,13){\framebox(3,3)}\put(27,-8){\framebox(3,3)}
\put(27,13){\line(0,-1){18}}\put(30,13){\line(0,-1){18}}
\qbezier(32,-7)(39,4)(32,14)
\put(33.5,12){\vector(-2,3){2}}\put(33.5,-5){\vector(-2,-3){2}}
\end{picture}$.

For each  Satake diagram with $\tau=\id$ we construct an even matrix that commutes with the elements of $\k$.
It is distinct from a scalar matrix which is the only adjoint invariant of $\g$.
For a Satake diagram  of shape $\Dg$ with $\tau\not = \id$ we construct an even  matrix that is $\k$-invariant under the twisted adjoint action:
$$
x\tp A\mapsto x A - A d x d, \quad x\in \k, \quad A\in \End(V),
$$
where  the matrix $d$ is the permutation $v_n\leftrightarrow v_{n'}$ that induces the flip of the tail roots in  the Dynkin diagram.
We find such a  matrix that is not proportional to $d$ (the only $\g$-invariants) with the restriction on the mixture parameters mentioned above.

Our approach to the task is as follows. We select a Satake sub-diagram $S^r\subset S$ of small rank that contains the tail of $S$,
and a shaft Satake  sub-diagram $S^l\subset S$. The rightmost block in $S^l$ is one of
(\ref{rightmost_block}) whose extension by $S^r$ complies with  the selection rules from Section \ref{SecDDD-SR}.
The selected sub-diagrams satisfy the requirement
$
S^l\cup S^r=S,
$
while their intersection $S^c=S^l\cap S^r$ is a shaft sub-diagram of total rank $\leqslant 2$ and white rank $\leqslant 1$:

$$
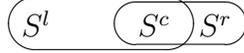
\begin{picture}(150,20)
\put(40,5){\oval(70,20)}\put(70,5){\oval(50,16)}
\put(10,0){$S^l$}\put(54,0){$S^c$}\put(77,0){$S^r$}
\end{picture}
$$
The diagram $S^r$ is chosen the smallest subject to these conditions.

We introduce subalgebras $\k^l,\k^c, \k^r\subset \k$ determined by these sub-diagrams and the vectors of mixture parameters restricted
from $\bar \Pi_\l$.
The subalgebras $\k^r$ and $\k^l$ defined this way are spherical with regard to $S^r$ and $S^l$, respectively.
We  proceed further as follows.
We construct a matrix invariant for the tail part subalgebra $\k^r$ first, and show that its restriction
to $\k^c$ has the block structure described by Corollary \ref{4-param-inv}. This allows for extension to $\k^l$-invariants and
thus to $\k$-invariants.

The structure of the result can be described by induction on the white rank $\ell$ of a growing subdiagram $S_\ell\subset S$, as in the proof of Proposition \ref{shaft_invariants}. Let $\k_\ell$ be the spherical subalgebra determined by $S_\ell$. The base is provided by $\k^r$-invariants.
Suppose that an invariant $\mathcal{I}_\ell$ has been constructed for $\k_\ell$, and let $\al\notin S_\ell$ be the next white root. Set $S_{\ell+1}=D(\al)\cup S_\ell$, and put $m=2$ if $D(\al)\setminus S_\ell$ contains a black root and $m=1$ otherwise.
Set
$\si_\al=
\left(
\begin{array}{ccc}
  1 & 0 \\
  0 & -1
\end{array}
\right)
$
and $\si_\al=1$ if $m=1$. By Corollary \ref{4-param-inv}, extend $\mathcal{I}_\ell$ to
$$
\mathcal{I}_{\ell+1}=(\la+\mu)\id_m+a\si_{\ell+1}+a'\si^t_{\ell+1}+\mathcal{I}_\ell,
$$
where $\mathcal{I}_\ell$ is the central block,
$\si_{\ell+1}=\si_\al$ and $\si_{\ell+1}^t$ are, respectively, the upper and lower skew-diagonal blocks, $\id_m$ is
the unit matrix block of size $m$ on the principal diagonal, and the parameters
$\la,\mu,a,a'$, with $aa'=-\la\mu$, depend on the mixture parameters of $\k_{\ell+1}$.

Thus, the proof reduces to the computing invariants of a shaft subalgebra  $\k^c$ which is done before,  and of  $\k^r$, which is done by  a case study.
The diagrams $S^r$ are those with  low "white rank" $\ell$ specified in Sections \ref{Sec_Black_tail}, \ref{Sec-White-tail}, and
\ref{Sec_Mixed_tail}.

In the case of black tails the diagrams of low $\ell$ are presented in (\ref{Black-tail}).
The sub-diagram $S^c$ of rank $0,0,1,1,2,1,1,2$, respectively,  comprises the nodes on the left of $\al$.

In diagrams with white tail listed in Section \ref{Sec-White-tail}, the diagram $S^c$ is the complement to the tail sub-graph $D_\t$.

In diagrams with mixed coloured tail from Section \ref{Sec_Mixed_tail}, the diagram $S^c$ consists of one node $\{\al_{n-3}\}$
for the first two diagrams in (\ref{mixed_tails_smallest}) and of two nodes $\{\al_{n-4},\al_{n-3}\}$ for the right diagram.

Thus we arrive at the main finding of this study.
\begin{thm}
  Spherical subalgebras $\k$ classified by the graded Satake diagrams are non-trivial for any vector $\vec c$ of non-zero  mixture parameters
  subject to the condition $c_{\tau(\al)}=c_\al$, with an appropriate normalization of root vectors.
\end{thm}
\noindent
This result substantiates the classification of Satake diagrams undertaken in this exposition.

\vspace{20pt}

\noindent
\underline{\large \bf Acknowledgement}

\vspace{10pt}
\noindent
This work is done at the Center of Pure Mathematics MIPT. It is financially supported by Russian Science Foundation grant 26-11-00115.

\noindent
D. Algethami is  thankful to the Deanship of Graduate Studies and Scientific Research at University of Bisha for supporting this work through the Fast-Track Research Support Program.

\vspace{10pt}
\noindent
\underline{\large \bf Data Availability}

\vspace{10pt}

 Data sharing not applicable to this article as no datasets were generated or analysed during the current study.

\subsection*{Declarations}

\underline{\large \bf Competing interests}

\vspace{10pt}
\noindent
The authors have no competing interests to declare that are relevant to the content of this article.

 \end{document}